\documentclass{amsart}
\usepackage{amsmath, amssymb, bbm}
\usepackage{color}
\usepackage{graphicx, epstopdf}

\usepackage{todonotes}

\theoremstyle{plain}
\newtheorem{theorem}{Theorem}[section]
\newtheorem{proposition}[theorem]{Proposition}
\newtheorem{corollary}[theorem]{Corollary}
\newtheorem{lemma}[theorem]{Lemma}

\newtheorem{assumption}[theorem]{Assumption}

\theoremstyle{definition}
\newtheorem{definition}[theorem]{Definition}

\theoremstyle{remark}
\newtheorem{remark}[theorem]{Remark}
\theoremstyle{remark}

\numberwithin{equation}{section}

\DeclareMathOperator{\dist}{dist}

\newcommand{\1}{\mathbf{1}}

\def \N {\mathbb{N}}
\def \R {\mathbb{R}}
\def \C {\mathbb{C}}

\def \Z {\mathbb{Z}}
\def \E {\mathbb{E}}
\def \P {\mathbb{P}}

\def \EE {\mathcal{E}}

\def \NN {\mathcal{N}}

\def \a {\mathfrak{a}}
\def \b {\mathfrak{b}}

\def \eps {\varepsilon}

\def \< {\langle}
\def \> {\rangle}
\def \^ {\widehat}

\def \dist {{\rm dist}}

\def \Prob {{\mathbb{P}}}

\def \im {{\rm Im}}
\def \re {{\rm Re}}

\def \supp {{\rm supp}}

\def \Comp {{\mathrm{Comp}}}
\def \Incomp {{\mathrm{Incomp}}}

\newcommand\Bb{{\mathbf b}}

\newcommand\Be{{\mathbf e}}

\newcommand\Bp{{\mathbf p}}

\newcommand\Bu{{\mathbf u}}
\newcommand\Bv{{\mathbf v}}
\newcommand\Bw{{\mathbf w}}
\newcommand\Bx{{\mathbf x}}
\newcommand\By{{\mathbf y}}
\newcommand\Bz{{\mathbf z}}

\newcommand\BX{{\mathbf X}}
\newcommand\BY{{\mathbf Y}}

\newcommand\BBu{\underline{\mathbf{u}}}

\begin{document}

\title[Eigenvectors and controllability]{Eigenvectors and controllability of non-Hermitian random matrices and directed graphs}

\author{Kyle Luh \and Sean O'Rourke}

\address{
  Center of Mathematical Sciences and Applications, Harvard University}
\email{kluh@cmsa.fas.harvard.edu}

\address{Department of Mathematics, University of Colorado at Boulder, Boulder, CO 80309   }
\email{sean.d.orourke@colorado.edu}

\thanks{K. Luh has been supported in part by the National Science Foundation under Award
No. 1702533}

\thanks{S. O'Rourke has been supported in
part by NSF grants ECCS-1610003 and DMS-1810500.}

\date{\today}

\begin{abstract}
We study the eigenvectors and eigenvalues of random matrices with iid entries.  Let $N$ be a random matrix with iid entries which have symmetric distribution.  For each unit eigenvector $\mathbf{v}$ of $N$ our main results provide a small ball probability bound for linear combinations of the coordinates of $\mathbf{v}$.  Our results generalize the works of Meehan and Nguyen \cite{MN} as well as Touri and the second author \cite{OT,OT3,OT2} for random symmetric matrices.  Along the way, we provide an optimal estimate of the probability that an iid matrix has simple spectrum, improving a recent result of Ge \cite{ge2017eigenvalue}.  Our techniques also allow us to establish analogous results for the adjacency matrix of a random directed graph, and as an application we establish controllability properties of network control systems on directed graphs.   
\end{abstract}

\maketitle

\section{Introduction}

 Let $\Bu \in \mathbb{C}^n$ be a random vector uniformly distributed on the unit sphere.  It follows that $\Bu$ has the same distribution as 
 \[ \frac{1}{\sqrt{ \sum_{i=1}^n |\xi_i|^2 } } \left( \xi_1, \ldots, \xi_n \right)^{\mathrm{T}}, \]
 where $\xi_1, \ldots, \xi_n$ are independent and identically distributed (iid) standard complex Gaussian random variables.  From this representation one can prove that $\1^\mathrm{T} \Bu$ converges in distribution to a standard complex Gaussian random variable, where $\1 \in \mathbb{C}^n$ is the all-ones vector.  We refer the reader to the survey \cite{OVW} for additional properties of $\Bu$.  
 
 Let $N$ be a random matrix of size $n \times n$ whose entries are iid random variables.  When the entries of $N$ are iid copies of a standard complex Gaussian random variable, $N$ is rotationally invariant, and the individual eigenvectors of $N$ have the same distribution as $\Bu$ above.  When the entries of $N$ are non-Gaussian, much less is known about the distribution of the eigenvectors.  In view of the universality phenomenon in random matrix theory, it is natural to conjecture that some of the properties that $\Bu$ possesses should also hold for the eigenvectors of $N$.  
 
In this note, we quantify some of these properties of the eigenvectors for iid random matrices.  The properties we focus on in this note are motivated by control theory, which we discuss in more detail in Section \ref{sec:ct} below.  

Eigenvectors of random matrices have been heavily studied in the last few years.  We refer the reader to \cite{MR3494576,MR3039372,MR3660521,MR2981427,MR3227063,MR3433288,MR3085669,Rlecturenotes,RVdeloc,CR,luh2018eigenvector,MR3034787,MR3256861,MR2525652,PhysRevLett.81.3367,MR3690289,BGZ,MR2839984,OVW,MR2558268,MR3449389,MR3183577,MR3164751,MR3129806,MR3039394,RVnogaps,RVnogaps,OT2,OT,OT3,MR2330979,LT,MR3851824,MR1003705,MR2846669,MR2782623,MR2930379,VuWangProjections,MR2537522,MR3025715,BNST,MR3755583,MC,MR3470349,BD,MR2782201,MR3606475,MR3955544,MN} and references therein for results concerning Hermitian and non-Hermitian random matrices. The results in \cite{MN, OT, OT2, OT3} are the most closely related to the present work.  
The following result is established by Meehan and Nguyen in \cite{MN}.  

\begin{theorem}[Follows from Theorem 1.5 in \cite{MN}] \label{thm:MN}
Let $\xi$ be a real-valued symmetric random variable with mean zero and unit variance so that
\[ \Prob( |\xi| \geq t) \leq K_1 \exp\left({-t^2/K_2} \right) \quad \text{ for all } t > 0 \]
for some constants $K_1, K_2 > 0$.  Let $W = (w_{ij})$ be an $n \times n$ real symmetric random matrix whose entries $w_{ij}$, $1 \leq i \leq j \leq n$ are iid copies of $\xi$.  Then there exist constants $C, \delta, \delta' > 0$ such that
\[ \Prob( \exists \text{ unit eigenvector } \Bv \text{ of } W \text{ such that } |\1^{\mathrm{T}} \Bv| \leq \eps ) \leq C \left(  n^{\delta} \eps + e^{-n^{\delta'}} \right) \]
for all $\eps > 0$, where $\1$ is the all-ones vector.    
\end{theorem}

Similar results are also established in \cite{OT, OT2, OT3}, and the results in \cite{MN} greatly generalize the results in \cite{OT2}.  In fact, the results in \cite{MN} are more general than what is stated in Theorem \ref{thm:MN} and apply to a large class of vectors (not just the all-ones vector).  

Intuitively, Theorem \ref{thm:MN} provides a non-asymptotic bound which shows that the eigenvectors have a similar behavior as the uniform vector $\Bu$ introduced above.  The goal of this work is to establish a version of Theorem \ref{thm:MN} for non-Hermitian random matrices.  Indeed, all the results in \cite{MN, OT, OT2, OT3} only apply to Hermitian random matrices.  When the random matrix is no longer Hermitian, the eigenvectors need not be orthogonal and new difficulties arise.  In this note, we develop upon the techniques introduced by Ge \cite{ge2017eigenvalue} in order to overcome these difficulties.  

\subsection{Notation}

Before stating our main results, we introduce some notation.  For a matrix $M$, we let $\|M\|$ denote the operator norm.  $M^{\mathrm{T}}$ is the transpose and $M^\ast$ is the conjugate transpose of $M$.  We write $M - z$ to denote $M - zI$, where $I$ is the identity matrix.  $J$ will denote the all-ones matrix.  For any square matrix, we will use the term eigenvector to denote a \emph{unit} eigenvector unless stated otherwise.  

We use bold letters to denote complex and real vectors.  For a vector $\Bv$, $\|\Bv\|$ is the Euclidean norm.  For two vectors $\Bv = (v_i)_{i=1}^n \in \mathbb{C}^n$ and $\Bu = (u_i)_{i=1}^n \in \mathbb{C}^n$, we let $\Bv \odot \Bu$ denote the Hadamard product of $\Bv$ and $\Bu$ defined as the vector $\Bv \odot \Bu = ( v_i u_i)_{i=1}^n$.  $\1$ denotes the all-ones vector.  

We use asymptotic notation under the assumption that $n \to \infty$.  In particular, the notations $X_n = O(Y_n)$, $Y_n = \Omega(X_n)$, $X_n \ll Y_n$, or $Y_n \gg X_n$ denote the bound $|X_n| \leq C |Y_n|$ for some constant $C > 0$ independent of $n$ and all $n > C$.  If the constant $C$ depends on a parameter (e.g., $C = C_k$), we indicate this with subscripts (e.g., $X_n = O_k(Y_n)$).  The notation $X_n = o(Y_n)$ denotes the bound $|X_n| \leq c_n Y_n$ for some sequence $c_n$ that converges to zero as $n$ tends to infinity.  

In our proofs, we often use $C, C', c, c'$, etc. to represent universal positive constants that can change from line to line. $[n]$ denotes the discrete interval $\{1, \ldots, n\}$ and $B(z, s)$ denotes a ball of radius $s$ centered at $z$.

\subsection{Eigenvector results}
In our main results below we focus on non-Hermitian random matrices with iid entries.  

\begin{definition}[iid random matrix]
Let $\xi$ be a real-valued random variable.  We say the $n \times n$ matrix $N$ is an \emph{iid random matrix} with \emph{atom variable} (or \emph{atom distribution}) $\xi$ if the entries of $N$ are iid copies of $\xi$.  
\end{definition}

We will often assume that the atom variable $\xi$ has mean zero.  In addition, we will sometimes need to assume that $\xi$ is a symmetric random variable, i.e., that $\xi$ has the same distribution as $-\xi$.  In the most general case, we will only need the following assumption.
\begin{assumption} \label{assump:main}
Assume $\xi$ is a real-valued random variable.  In addition, assume there exists constants $q \in (0,1)$ and $T > 0$ so that
\begin{align}
	\sup_{u \in \mathbb{R}} \Prob ( |\xi - u| < 1) \leq 1-q, \label{assump:1}\\
	 \Prob( 1 \leq |\xi- \xi'| \leq T) \geq q/2,  \label{assump:2}
\end{align}
and
\begin{equation} \label{assump:3}
	\Prob(|\xi| > T) \leq q/2, 
\end{equation} 
where $\xi'$ is an independent copy of $\xi$.  
\end{assumption}
\begin{remark}
Assumption \eqref{assump:1} guarantees that $\xi$ is non-degenerate.  All three conditions \eqref{assump:1}, \eqref{assump:2}, and \eqref{assump:3} hold (for some $T$ and $q$) when $\xi$ has finite variance of at least $1$.  Many of our results will have constants that implicitly depend on $q$ and $T$.  We will suppress this dependence in the notation and statements of the theorems.
\end{remark}

Our first main result is the analogue of Theorem \ref{thm:MN} for iid random matrices.  

\begin{theorem} \label{thm:maineigvec}
Let $N$ be an $n \times n$ iid random matrix with real-valued symmetric atom variable $\xi$ which satisfies Assumption \ref{assump:main}, and let $K > 1$ be a constant.    Then there exist constants $C, c > 0$ (depending only on the constant $K$ and the atom variable $\xi$) such that 
\[ \Prob ( \exists \text{ unit eigenvector } \Bu \text{ of } N \text{ such that } |\1^\mathrm{T} \Bu| \leq t ) \leq C n t + \Prob( \|N \| > K \sqrt{n} ) \]
for any $t \geq e^{- c n}$.  Here $\1$ denotes the all-ones vector.  
\end{theorem}


A bound on the operator norm can be controlled by additional moment assumptions on $\xi$.  For instance, when $\xi$ has finite fourth moment there exists $K > 1$ so that 
\begin{equation} \label{eq:4momnorm}
	\Prob( \|N \| > K \sqrt{n} ) = o(1), 
\end{equation} 
and when $\xi$ satisfies a sub-Gaussian assumption 
\[ \Prob( \|N \| > K \sqrt{n} ) \leq C e^{-cn}, \]
where the constants and rate of convergence in these bounds depend on the fourth moment or sub-Gaussian moment of $\xi$ (see \cite{MR2963170} and \cite{MR950344}).  

More generally, we have the following theorem.
\begin{theorem} \label{thm:maineigvecscaled}
	Let $N$ be an $n \times n$ iid random matrix with real-valued symmetric atom variable $\xi$ which satisfies Assumption \ref{assump:main}, and let $B, K > 1$ be constants.  Then there exist constants $C, c, \nu, \nu' > 0$ (depending only on the constants $K, B$ and the atom variable $\xi$) such that the following holds.  Let $m \leq \nu \sqrt{n}$ and $\Bb \in \C^n$ be a vector such that $B^{-1} \leq |b_i| \leq B$ for all but $m$ coordinates of $\Bb$.    
	Then
	\[ \Prob ( \exists \text{ unit eigenvector } \Bu \text{ of } N \text{ such that } |\Bb^\mathrm{T} \Bu| \leq t ) \leq Cn t + \Prob( \|N \| > K \sqrt{n} ) \]
	for any $t \geq e^{- \nu' n/m}$.  
\end{theorem}


\subsection{Eigenvalue Gaps}
Tail bounds between gaps of eigenvalues of random matrices were originally studied in \cite{arous2013extreme} in the GUE case and in \cite{NTV, TVsimple} for a large class of Hermitian random matrices.  In his thesis \cite{ge2017eigenvalue}, Ge proves a similar result for iid matrices.  
Let $\lambda_1(N), \dots, \lambda_n(N)$ be the eigenvalues of a matrix $N$.  Let $\Delta \equiv  \Delta(N) := \min_{i \neq j} |\lambda_i(N) - \lambda_j(N) |$.  Ge obtained the following theorem.
\begin{theorem}[Theorem 3.1.1, \cite{ge2017eigenvalue}] \label{thm:eiggaps}
	Let $N$ be an $n \times n$ iid random matrix whose atom variable satisfies Assumption \ref{assump:main} and has mean zero.  
	For every $C > 0$ and $\delta \geq s \geq n^{-C}$
	$$
	\P(\Delta(N) < s \sqrt{n}) = O\left(\delta n^{2 + o(1)} + \frac{s^2 n^{4 + o(1)}}{\delta^2} \right) + \P(\|N\| \geq K \sqrt{n})
	$$ 
	where the implied constant depends only on the parameters in Assumption \ref{assump:3} and $C$.  
\end{theorem} 

One immediate consequence is that with high probability, the random matrix has simple spectrum.
\begin{corollary}[\cite{ge2017eigenvalue}]
	Let $N$ be an $n \times n$ iid random matrix whose atom variable has mean zero, unit variance, and finite fourth moment.  
	Then
	$$
	\P(\Delta(N) = 0) = o(1).  
	$$
\end{corollary} 

Building on the techniques in \cite{ge2017eigenvalue}, we greatly extend the range of the tail bound for the eigenvalue gaps and also improve the probability bound for simple spectrum.
\begin{theorem} \label{thm:eiggaptail}
	Let $N$ be an $n \times n$ iid random matrix whose atom variable satisfies Assumption \ref{assump:main}.  
	Then there exist constants $C, c > 0$ such that for $s \geq 0$,
		$$
	\P(\Delta(N) \leq s \text{ and } \|N \| \leq K \sqrt{n} ) \leq C s^{2/3} n^5 + Ce^{-c n} + \P(\|N \| \geq K \sqrt{n} ). 
	$$
\end{theorem}
While the right-hand side appears non-optimal, we can deduce an immediate corollary.
\begin{corollary}
	If in addition to the assumptions of Theorem \ref{thm:eiggaptail} we assume the entries of $N$ are subgaussian with mean zero, then there exist constants $C,c > 0$ such that
	$$
	\P(\Delta(N) = 0 ) \leq Ce^{-c n}.  
	$$ 
\end{corollary}

This corollary is of independent interest and clearly optimal up to the constants for subgaussian entries, while Ge's result only guarantees a polynomially small probability.  The simple spectrum probability bound is also an important technical tool for the results of the next section.

For a directed graph $G = ([n], E)$ with vertex set $[n]$ and edge set $E$, we let the adjacency matrix $A$ be defined by 
$$
A_{ij} = \begin{cases} 1, & \text{if } (i,j) \in E, \\
0, & \text{otherwise}.
\end{cases}
$$ 
We define the directed Erd\H{o}s--R\'enyi random graph to be the random digraph on vertex set $[n]$ such that each edge $(i,j)$ appears independently with probability $p$, for a constant $p \in (0, 1)$.  The adjacency matrix $A$ is random but does not fall under the purview of Theorem \ref{thm:eiggaptail} as $\|A\| = \Omega(n)$ with high probability (so $\P(\| A \| > K \sqrt{n} ) = 1 -o(1)$).  In addition, our results apply to both the model where loops are allowed (so that $(i,i)$ is an edge with probability $p$) as well as the case where loops are not allowed (so that the adjacency matrix has zeros along the diagonal with probability one).   For the adjacency matrix $A$ for either model, we are able to prove the following weaker conclusion.

\begin{theorem} \label{thm:adjeiggaptail}
	There exist constants $C, c > 0$, depending only on $p$, such that
	$$
	\P(\Delta(A) = 0 ) = o(1). 
	$$
\end{theorem}

\subsection{Connection to control theory} \label{sec:ct}

Our main results are related to a large collection of works on controllability of network control systems \cite{MR3353397,MR3314353,MR3000441,Liu:2011aa,MR2972713,MR3526189,MR3152278,OT, OT3,OT2,MR2480130,1428782}.  
Unlike many of these previous works, in this note we take a stochastic approach.  In this section we provide a brief overview of linear control theory and its connection to our main results above.  For additional details concerning control of linear systems, the reader is advised to see \cite{MR3000441,Kls} and references within.  

We consider a discrete-time linear state-space system formed from an $n \times n$ matrix $A$ (called the state transition matrix) and a vector $\Bb \in \mathbb{R}^n$ (the given input vector).  The system's state at time $k$ is a vector $\Bx(k)$ which evolves according to the constraint: 
\[ \Bx(k+1) = A \Bx(k) + u(k) \Bb, \]
where each $u(k)$ is a scalar.  The sequence $(u(k))_{k \geq 0}$ is the control of the system.  

Roughly speaking, the system is controllable if we can find the control values $u(\cdot)$ based on arbitrary state values $\Bx(\cdot)$.  Following \cite{MR3000441, MN} we observe that since
\begin{align*}
	\Bx(1) &= A \Bx(0) + u(0) \Bb \\
	\Bx(2) &= A \Bx(1) + u(1) \Bb = A^2 \Bx(0) + u(0) A \Bb + u(1) \Bb \\
	&\vdots \\
	\Bx(n) &= A^n \Bx(0) + u(0) A^{n-1} \Bb + u(1) A^{n-2} \Bb + \cdots + u(n-1) \Bb,
\end{align*}
it follows that
\begin{equation} \label{eq:ct}
	\Bx(n) - A^n \Bx(0) = \begin{pmatrix} A^{n-1} \Bb & A^{n-2} \Bb & \cdots & A\Bb & \Bb \end{pmatrix} \begin{pmatrix} u(0) \\ u(1) \\ \vdots \\ u(n-1) \end{pmatrix}. 
\end{equation} 
Thus, we can find the control values $u(\cdot)$ based on the state values $\Bx(\cdot)$ if and only if the matrix on the right-hand side of \eqref{eq:ct} has full rank.  
This leads immediately to the following definition (known as Kalman's rank condition) for controllability.
\begin{definition} \label{def:kalman}
Let $A$ be an $n \times n$ matrix, and let $\Bb$ be a vector in $\mathbb{R}^n$.  We say the pair $(A,\Bb)$ is \emph{controllable} if the $n \times n$ matrix
\begin{equation} \label{eq:matrix}
	\begin{pmatrix} \Bb & A\Bb & \cdots & A^{n-1} \Bb \end{pmatrix} 
\end{equation} 
has full rank (that is, rank $n$).  Here the matrix in \eqref{eq:matrix} is the matrix with columns $\Bb$, $A\Bb$, \ldots, $A^{n-1} \Bb$.  We say $(A,\Bb)$ is \emph{uncontrollable} if it is not controllable.  
\end{definition}  

Given the state transition matrix $A$, two important problems are:
\begin{enumerate}
\item (Minimal controllability) What is the sparsest nonzero binary vector $\Bb \in \{0,1\}^n$ such that $(A,\Bb)$ is controllable?
\item (Uniform controllability) If $\1$ is the all-ones vector, is $(A,\1)$ controllable? 
\end{enumerate}

Our main results above allow us to study versions of these problems when $A$ is a random matrix.  Loosely speaking, our results show that ``most'' systems are controllable, which confirms a similar phenomenon that was observed previously for systems with Hermitian transition matrices \cite{MN, OT, OT2, OT3}.  In addition, we also consider the case when the vector $\Bb$ is random.  

As corollaries to our main results above, we obtain the following.  

\begin{corollary} \label{cor:allones}
Let $\xi$ be a real-valued symmetric random variable with mean zero, unit variance, and finite fourth moment. Let $N$ be the $n \times n$ iid random matrix with atom variable $\xi$.  Then $(N, \1)$ is controllable with probability $1 - o(1)$, where $\1$ is the all-ones vector.  
\end{corollary}

\begin{corollary} \label{cor:random}
	Let $\xi$ be a real-valued random variable with mean zero, unit variance, and finite fourth moment. Let $N$ be the $n \times n$ iid random matrix with atom variable $\xi$.  Let $\psi$ be a real-valued random variable that satisfies Assumption \ref{assump:main}, and assume $\Bb \in \mathbb{R}^n$ is a random vector with entries that are iid copies of $\psi$.  Then, with probability $1-o(1)$, $(N, \Bb)$ is controllable.  
\end{corollary}
	We note that Corollary \ref{cor:random} does not require symmetric random variables.
\begin{remark}
	If instead of a bounded fourth moment, we assume the entries of $N$ are subgaussian in Corollaries \ref{cor:allones} or \ref{cor:random}, we can improve the probability bound to $1 - Ce^{-c n}$ for some constants $C,c > 0$.  
\end{remark}

\begin{corollary} \label{cor:basis}
	Let $\xi$ be a real-valued random variable with mean zero, unit variance, and finite fourth moment. Let $N$ be the $n \times n$ iid random matrix with atom variable $\xi$.  Then
	\[ \inf_{1 \leq i \leq n} \P ( (N, e_i) \text{ is controllable}) = 1 - o(1). \] 
\end{corollary}

Corollary \ref{cor:allones} and \ref{cor:basis} answer the uniform controllability and minimal controllability questions from above for non-Hermitian random matrices.  

We have corresponding corollaries for the adjacency matrix of directed random graphs.

\begin{corollary} \label{cor:allonesgraph}
	Let $A$ be the $n \times n$ adjacency matrix of an Erd\H{o}s--R\'enyi directed graph with constant edge probability $p = 1/2$.  Then $(A, \1)$ is controllable with probability $1-o(1)$, where $\1$ is the all-ones vector.  
\end{corollary}

\begin{corollary} \label{cor:randomA}
		Let $A$ be the $n \times n$ adjacency matrix of an Erd\H{o}s--R\'enyi directed graph with constant edge probability $p \in (0,1)$.  Let $\psi$ be a real-valued random variable that satisfies Assumption \ref{assump:main}, and assume $\Bb \in \mathbb{R}^n$ is a random vector with entries that are iid copies of $\psi$.   Then, with probability $1-o(1)$, $(A, \Bb)$ is controllable.  
\end{corollary}

\begin{corollary} \label{cor:basisgraph}
		Let $A$ be the $n \times n$ adjacency matrix of an Erd\H{o}s--R\'enyi directed graph with constant edge probability $p \in (0,1)$.  Then
		\[ \inf_{1 \leq i \leq n} \P( (A, e_i) \text{ is controllable}) = 1- o(1). \]
\end{corollary}

Unlike Corollary \ref{cor:allonesgraph}, Corollaries \ref{cor:randomA} and \ref{cor:basisgraph} do not demand that $p = 1/2$.  
\subsection{Overview and outline}

In Section \ref{sec:approximatenull}, we isolate the key structural result which guarantees that any vector near the kernel of an iid random matrix (shifted by a complex number) is unstructured.  The investigation of the structure of vectors as they relate to their anti-concentration has a long history in random matrix theory beginning with the infamous singularity problem for discrete random matrices \cite{komlos1967determinant, kahn1995probability, tao2006random, costello2006random, bourgain2010singularity, nguyen2012inverse, ferber2019counting, ferber2019singularity, tikhomirov2020singularity}.  Strong bounds on the least singular value in both the symmetric and non-symmetric setting used similar tools \cite{edelman1988eigenvalues, spielman2004smoothed, rudelson2008LO, rudelson2009smallest,vershynin2014symmetric, tao2009inverse, tao2010smooth, luh2018complex, rebrova2018coverings, jain2019approximate, ge2017eigenvalue}. The quantitative estimates in Section \ref{sec:approximatenull} build on this rich history of anti-concentration in random matrix theory.  In particular, our quantitative estimates improve on those in \cite{ge2017eigenvalue}.  The proof uses a delicate covering argument to exclude structured vectors.  The primary obstacle that appears in the non-Hermitian setting is that the eigenvectors can now reside in the \emph{complex} unit sphere which has doubled the dimension of the space that must be covered.  The key geometric insight that resolves this issue is expounded on in Section \ref{subsec:nets}.  In Section \ref{sec:eigstruc}, we use an approximation argument to extend the structural result to eigenvectors of a non-Hermitian matrix.  We utilize a multi-scale argument to extend our structural result to small-ball probability bounds on all scales.     

The arguments in Sections \ref{sec:approximatenull} and \ref{sec:eigstruc} do  not immediately apply to the adjacency matrix of a random directed graph because the operator norm of the adjacency matrix is $\Omega(n)$ with high probability.  In Section \ref{sec:directedrandomgraphs}, we describe the method to generalize the structural result to directed graphs.  The key observation is that the matrix of expectations is low-rank so the covering arguments from the previous sections can be extended as the size of nets do not incur many new dimensions.  We then utilize previous results on the spectrum of rank-1 perturbations of random matrices which state that the eigenvalues of perturbed matrix are all contained in the centered disk with radius determined by the spectral norm of the unpertrubed matrix, except for one outlier.  To understand the structure of the eigenvector corresponding to the outlier, we use the Perron--Frobenius theorem for non-negative matrices.  In Section \ref{sec:scaledvectors}, we show that for a fixed vector $\Bb$, even $\Bb \odot \Bu$ has no structure, where $\odot$ denotes the Hadamard product and $\Bu$ is an eigenvector.  

Finally, in Section \ref{sec:controllability}, we complete the proofs of our main results and deduce the control theory corollaries from our eigenvector structure results.  To relate the structure of eigenvectors to the controllability of the matrix requires the introduction of auxiliary random signs in the matrix that preserve the distribution of the matrix and only alter the signs of the entries in the eigenvectors.  The first condition will require symmetric entries in the random matrix for some of the control theory results.    

In Appendix \ref{appendix:tailtwoeigs} we include the proof of Theorem \ref{thm:eiggaptail} and in Appendix \ref{appendix:adjtailtwoeigs} we complete the proof of Theorem \ref{thm:adjeiggaptail}. They are similar to previous arguments in this article and in \cite{ge2017eigenvalue}.   

\subsection*{Acknowledgements}
We thank Hoi H. Nguyen for pointing out reference \cite{ge2017eigenvalue}.  The second author thanks Behrouz Touri for introducing him to the problem and answering numerous questions.

\section{Arithmetic Structure  of Approximate Null Vectors} \label{sec:approximatenull}

In this section, we study the arithmetic structure of approximate eigenvectors.  We let $\EE_K$ denote the event that $\|N\| \leq K \sqrt{n}$.  The goal of this section is to prove the following result.
\begin{theorem} \label{thm:nullvectors}
	Let $N$ denote the $n \times n$ matrix with entries that are iid copies of a random variable $\xi$ that satisfies Assumption \ref{assump:main}.
	There exist constants $c_{\ref{thm:nullvectors}}, c_{\ref{thm:nullvectors}}', c_{\ref{thm:nullvectors}}'', c_{\star}, \mu > 0$ such that the following holds.
	We let $M$ denote the matrix $N - \lambda \sqrt{n} I$ where $\lambda$ is a fixed complex number with $|\lambda| \leq K$ and $\delta = \im(\lambda) \geq e^{-c_{\star} n}$.  
	If $$c_{\ref{thm:nullvectors}}' \sqrt{n}/\delta \leq \widetilde{D} \leq e^{c_{\ref{thm:nullvectors}}'' n}$$
	then with probability at least $1 - e^{-c_{\ref{thm:nullvectors}} n}$, on the event $\EE_K$, any complex vector, $\Bz \in S_{\C}^{n-1}$, such that $\|M \Bz\| \leq K \mu n/ \widetilde{D}$ has $d(\Bz) \geq c_{\ref{lem:lowercorrelation} }\delta$ and $D(\Bz) \geq \widetilde{D}$.
\end{theorem}
	$d(\Bz)$ and $D(\Bz)$ denote the real-imaginary correlation and the LCD respectively and are defined formally in Definitions \ref{def:complexvector} and \ref{def:LCD} below.
  Some aspects of the proofs below are inspired by arguments from \cite{RVnogaps, ge2017eigenvalue}, but we have introduced several modifications and novelties to handle our current setting.

\begin{definition}
	For two constants $a, b \in (0,1)$, we say a vector $\Bx \in S^{n-1}_{\C}$ is \emph{compressible} if there is a $an$-sparse vector $\Bx'$ such that $\|\Bx - \Bx'\| \leq b$.  We denote the set of compressible vectors as $\Comp_{\C}(a,b)$. Let $\Incomp_{\C}(a,b)$ denote the \emph{incompressible} vectors, which are those on the unit sphere that are not compressible.  The same definitions apply to real vectors, in which case, we use $\Comp_{\R}$ and $\Incomp_{\R}$.
\end{definition}

The following is a well-known result that follows from tensorizing a crude estimate for fixed vectors and taking a union bound.
\begin{lemma} \label{lem:comp}
	There exist constants $a, b, c_{\ref{lem:comp}} \in  (0, 1)$ and $K > 2$ such that 
	$$
	\P\left(\inf_{\Bz \in \Comp(a,b)} \|M \Bz\| \leq c_{\ref{lem:comp}} \sqrt{n} \text{ and } \EE_K \right) \leq e^{-c_{\ref{lem:comp}} n}.
	$$ 
\end{lemma}

We fix the constants $a, b, c_{\ref{lem:comp}}$ for the remainder of the argument.  The next lemma from \cite{ge2017eigenvalue} demonstrates that an approximate null-vector cannot have mass exclusively confined to the real or imaginary parts.

\begin{lemma} \label{lem:lowerboundrandi}
	Let $\Bz \in S^{n-1}_{\C}$ be incompressible and $\Bz = \Bx + i \By$ with $\Bx, \By \in \R^n$. There exists a constant $c_{\ref{lem:lowerboundrandi}}$ such that on the event $\EE_K$, if $\|M \Bz\| \leq c_{\ref{lem:lowerboundrandi}} \delta \sqrt{n}$ then $\|\Bx \| \geq c_{\ref{lem:lowerboundrandi}} \delta$ and $\|\By\| \geq c_{\ref{lem:lowerboundrandi}} \delta$.  
\end{lemma}
\begin{proof}
	Let $M = N + i \delta \sqrt{n} I$ where $N  \in \R^{n \times n}$.  By examining the real part of $\|M \Bz\| \leq c \delta \sqrt{n}$, we must also have that $\|N \Bx - \delta \sqrt{n} \By\| \leq c \delta \sqrt{n}$.  This implies that 
	$$
	\|\By\| \leq \frac{\|N \Bx\| + c \delta \sqrt{n}}{\delta \sqrt{n}} \leq \frac{2 K \sqrt{n} \|\Bx\| + c \delta \sqrt{n} }{\delta \sqrt{n}}.
	$$  
	Therefore, as $\Bz$ is a unit vector,
	\begin{align*}
	1 &= \|\Bx\|^2 + \|\By\|^2  \\
	&\leq \|\Bx\|^2 + \left(\frac{2 K \sqrt{n} \|\Bx\| + c \delta \sqrt{n} }{\delta \sqrt{n}} \right)^2 \\
	&\leq \|\Bx\|^2 + \frac{ 8 K^2 }{\delta^2 } \|\Bx\|^2 + 2 c^2
	\end{align*}
	From the above, we can conclude that
	$$
	\|\Bx\|^2 \geq \frac{1  - 2c^2}{1 + \frac{8 K^2}{\delta^2}} \geq c' \delta 
	$$
	for a small enough $c$ and $c'$, depending on $K$.  Finally, we can set $c_{\ref{lem:lowerboundrandi}}$ to be the smaller of $c$ and $c'$.
\end{proof}

\begin{remark}
	Note that $\|M \Bz\| = \|M e^{i \theta} \Bz\|$ so the above lemma applies to any rotation of $\Bz$.  
\end{remark}

\subsection{Excluding Vectors with Real Compressible Part}
\begin{lemma} \label{lem:realcomp}
	Let $ \alpha \in [c_{\ref{lem:lowerboundrandi}} \delta, 1/2]$.  There exist constants $\a, \b, c_{\ref{lem:realcomp}}$ such that for 
	$$
	S_{\alpha} := \left \{\Bz = \Bx+ i \By \in \Incomp_{\C}(a,b): \alpha < \|\Bx\| \leq 2 \alpha, \frac{\Bx}{\|\Bx\|_2} \in \Comp_{\R}(\a, \b) \right \}
	$$
	we have
	$$
	\P\left(\inf_{z \in S_{\alpha}} \|M \Bz\| \leq \b \delta  \sqrt{n} \text{ and } \mathcal{E}_K \right) \leq e^{-c_{\ref{lem:realcomp}} n}.
	$$
\end{lemma}
\begin{proof} 
	{\bf Case I: }  We assume that $\alpha \geq C \delta$ where $C$ is a large constant to be determined.  Recall that we let $M = N + \lambda \sqrt{n} I$.  Again, we examine the real part of the inequality $\|M \Bz\| \leq \b \delta \sqrt{n}$ which implies that $\|N \Bx - \delta \sqrt{n} \By\| \leq \b \delta \sqrt{n}$.  Let 
	$$
	T_{\alpha} := \left \{ \Bx \in \R^n: \alpha < \|\Bx\| \leq 2 \alpha, \frac{\Bx}{\|\Bx\|} \in \Comp_{\R}(\mathfrak{a}, \mathfrak{b}) \right\}.
	$$
	To complete the proof in this case, it suffices to show that 
	$$
	\P\left(\inf_{\Bx \in T_{\alpha}} \|N \Bx\| \leq  (1 +\b) \delta \sqrt{n} \text{ and } \EE_K \right) \leq e^{-c n}.
	$$
	The intuition is that as $\Bx$ is close to sparse, $\|N \Bx\|$ should be on the order of $\|\Bx\| \sqrt{n}$.  Thus, choosing $\|\Bx\| \geq C \delta$ for large enough $C$ should violate the event $\|N \Bx\| \leq (1 + \b) \delta \sqrt{n}$ with high probability.  
	Let $\mathcal{N}'$ be a $\mathfrak{b}$-net of $\Comp_{\R}(\mathfrak{a}, \mathfrak{b})$.  By the standard volumetric argument, we can construct $\mathcal{N}'$ so that $|\mathcal{N}'| \leq \binom{n}{\mathfrak{a} n}(3/\mathfrak{b})^{\mathfrak{a} n}$.  Now let $\mathcal{N}''$ be an $\alpha \mathfrak{b}$-net of the interval $[\alpha, 2 \alpha]$.  Clearly, we can have $|\mathcal{N}''| \leq \frac{2}{\mathfrak{b}}$.  Finally, let
	$$
	\mathcal{N} := \{a' \Bx': a' \in \mathcal{N}'' \text{ and } \Bx' \in \mathcal{N}' \}.
	$$
	We have $|\mathcal{N}| \leq \binom{n}{\mathfrak{a} n}(3/\mathfrak{b})^{\mathfrak{a} n} \frac{2}{\mathfrak{b}}$.  Furthermore, for $\Bx \in T_{\alpha}$, there exists $a'$ and $\Bx'$ such that $|a' - \|x\|| \leq \alpha \mathfrak{b}$ and $\left\|\Bx' - \frac{\Bx}{\|\Bx\|} \right\| \leq \mathfrak{b}$ so
	\begin{align*}
	\|\Bx - a' \Bx'\| &\leq \Big\|\Bx - \|\Bx\| \Bx'\Big\| + \Big\|\|\Bx\| \Bx' - a' \Bx'\Big\| \\
	&\leq \alpha \mathfrak{b} + \alpha \frak{b}.
	\end{align*}
	Therefore, $\mathcal{N}$ is a $2 \alpha \mathfrak{b}$-net of $T_{\alpha}$.  
	
	By the standard tensorization argument (c.f.	\cite[Lemma 3.2]{rudelson2009smallest}), we have that for $\Bx \in S_{\R}^{n-1}$, there exists a small constant $c > 0$ such that
	$$
	\P(\|N \Bx\| \leq c \sqrt{n}) \leq e^{-c n}.
	$$
	Therefore, by a simple union bound,
	$$
	\P\left(\inf_{\Bx \in \mathcal{N}} \|N \Bx\| \leq 2 c \alpha \sqrt{n}\right) \leq \binom{n}{\mathfrak{a} n}\left( \frac{3}{\mathfrak{b}} \right)^{\mathfrak{a} n} \frac{2}{\mathfrak{b}} e^{-c n} \leq e^{-c' n}
	$$
	for a small constant $c'>0$ after choosing $\a$ small enough.  For any $\Bx \in T_{\alpha}$, there exists $\Bx' \in \NN$ such that $\|\Bx - \Bx'\| \leq 2 \alpha \b$.   On the event that $\inf_{\Bx \in \NN} \|N \Bx\| \geq 2 c \alpha   \sqrt{n}$,  
	\begin{align*}
	\|N \Bx\| &\geq \|N \Bx'\| - \|N\| \|\Bx - \Bx'\| \\
	&\geq 2 c \alpha   \sqrt{n} - K \sqrt{n} 2 \alpha \b \\
	&\geq 2 ( c- K \b ) C \delta \sqrt{n}.
	\end{align*}
	Choosing $\b$ small enough so that $c - K \b > 0$ and then choosing $C$ large enough, we have that this implies that
	$$
	\|N \Bx\| > (1 + \b) \delta \sqrt{n}. 
	$$
	Therefore,
	$$
	\P(\inf_{\Bx \in T_{\alpha}} \|N \Bx\| \leq  \b \delta \sqrt{n}) \leq e^{-c' n}.
	$$
	
	{\bf Case II: }
	We utilize the real and imaginary parts of the inequality $\|M \Bz\| \leq \b \delta \sqrt{n}$ with $\Bz = \Bx + i \By$. We must have
	$$
	\|N \Bx - \delta \sqrt{n} \By\| \leq \b \delta \sqrt{n}
	$$
	and
	$$
	\|N \By + \delta \sqrt{n} \Bx\| \leq \b \delta \sqrt{n}.
	$$
	Let us define for an index set $I \subset [n]$ with $|I| = \a n$, 
	$$
	T_{\alpha, I} := \left \{ \Bx \in \R^n: \alpha < \|\Bx\| \leq 2 \alpha, \frac{\Bx}{\|\Bx\|} \in \Comp_{\R}(\mathfrak{a}, \mathfrak{b}), \supp(x) \subset I \right\}.
	$$
	For concreteness, let us assume for now that $I = \{1, \dots, \a n\}$.
	Similar to Case I, we can find a $2 \b \alpha$-net, $\mathcal{N}$, of $T_{\alpha, I}$ such that $|\NN| \leq \left( \frac{3}{\mathfrak{b}} \right)^{\mathfrak{a} n} \frac{2}{\mathfrak{b}}$.  Conditioning on the first $\mathfrak{a} n$ columns of $M$, we have the deterministic inequality
	$$
	\left\|\frac{N \Bx}{\delta \sqrt{n}} - \By\right\| \leq \b.
	$$  
	We construct a random net, depending on the first $\mathfrak{a} n$ columns of $M$, for the imaginary part of the vectors.  We use $ \frac{N \Bx}{\delta \sqrt{n}}$ to approximate the imaginary part of the complex vectors in $S_{\alpha}$.  Note that since $\Bx$ is only supported on the first $\a n$ coordinates, $\frac{N \Bx}{\delta \sqrt{n}}$ depends on only the first $\a n$ columns of $M$. Define 
	$$
	\NN' := \left\{\Bx + \frac{N \Bx}{\delta \sqrt{n}} i: \Bx \in \NN \right\}.
	$$
	Therefore, on the event that $\|M \Bz\| \leq c_{\ref{lem:lowerboundrandi}} \delta \sqrt{n}$ with $\Bz = \Bx+i\By$, for $\Bx' \in \NN$ such that $\|\Bx - \Bx'\| \leq 2 \b \alpha$, we define $\By'= \frac{N \Bx'}{\delta \sqrt{n}}$ so that
	\begin{align*}
	\left\|\By' - \By \right\| &\leq \left \|\frac{N \Bx'}{\delta \sqrt{n}} - \frac{N \Bx}{\delta \sqrt{n}} \right\| + \left\| \frac{N \Bx}{\delta \sqrt{n}} -  \By\right\|  \\
	&\leq \frac{2 K \mathfrak{b} \alpha}{\delta} + \b  \\
	&\leq C' \mathfrak{b}
	\end{align*}
	for some large constant $C'$
	where in the last line we have used the assumption that $\alpha \leq C \delta$.  
	Since $\Bz$ is incompressible and $\Bx$ is compressible, we must have that $\|\By\| \geq \frac{b}{2}$ after reducing $\mathfrak{b}$ if necessary.  We write 
	$$
	\By' = \left( \begin{array}{c}
	\By_1 \\
	\By_2 \end{array} \right)
	$$  
	where $\By_1$ is the vector formed by the first $\mathfrak{a} n$ coordinates and $\By_2$ are the remaining coordinates.  Since $\Bz$ is incompressible and $\left\|\By' - \By \right\| \leq C \mathfrak{b}$, we can choose $\mathfrak{b}$ small enough such that $\|\By_2\| \geq b/4$.  By the standard tensorization argument,
	$$
	\P \left( \|N \By' +  \delta \sqrt{n} \Bx' \| \leq c \sqrt{n} \right) \leq e^{-c n}
	$$
	where the probability is taken over the randomness of the last $n - \a n$ columns of $M$ and the lower bound on $\|\By_2\|$.  Thus, by a union bound,
	$$
		\P \left( \inf_{\Bz' \in \mathcal{N}'} \|N \By' + \delta \sqrt{n} \Bx'\| \geq  c  \sqrt{n} \right) \leq \left( \frac{3}{\mathfrak{b}} \right)^{\mathfrak{a} n} \frac{2}{\mathfrak{b}} e^{-c n} \leq e^{-c' n}.
	$$

	On the event that $\inf_{\Bz' \in \mathcal{N}'} \|N \By' + \delta \sqrt{n} \Bx'\| \geq  c  \sqrt{n}$, for any $\Bz = \Bx + i \By \in T_{\alpha, I}$,
	\begin{align*}
	\|N \By + \delta \sqrt{n} \Bx\| &\geq \|N \By' + \delta \sqrt{n} \Bx'\| -  \|N (\By' - \By) \| - \|\delta \sqrt{n} (\Bx' - \Bx)\| \\
	&\geq c  \sqrt{n} - K C' \sqrt{n} \mathfrak{b} - 2 \delta \sqrt{n} \b \alpha \\
	&\geq c'' \sqrt{n}
	\end{align*}
	after reducing $\b$ if necessary.  Finally, taking a union bound over the $\binom{n}{\a n}$ possible $I$ and then choosing $\a$ small enough shows that
	$$
	\P\left( \inf_{\Bz \in S_{\alpha}} \|M \Bz \| \leq \b \delta \sqrt{n} \text{ and } \EE_K \right) \leq e^{-c_{\ref{lem:realcomp}}n}
	$$
	for a small enough $c_{\ref{lem:realcomp}}$.
\end{proof}

\subsection{LCD and Structure Theorem}
We import several definitions to quantify the structure, or absence thereof, of a vector, a matrix and a complex vector.  
\begin{definition}  \label{def:LCD}
	The following notions were developed in a series of papers by Rudelson and Vershynin \cite{rudelson2008LO, rudelson2009smallest, vershynin2014symmetric, RVnogaps}.
\begin{itemize}
	\item For a vector $\Bv \in \R^n$, we define the least common denominator (LCD) of $\Bv$ to be
	$$
	D(\Bv) = D(\Bv; L, \rho) := \inf \left\{ \theta \in \R^+: \dist(\theta \Bv, \Z^n) <  \rho L \sqrt{ \log_+ \frac{\|\theta \Bv\|}{L}} \right \} 
	$$
	\item For a matrix $U \in \R^{m \times n}$, we define the LCD of $U$ to be
	$$
	D(U) = D(U; L, \rho) := \inf \left \{ \|\boldsymbol{\theta}\|: \boldsymbol{\theta} \in \R^m, \dist(U^{\mathrm{T}} \boldsymbol{\theta}, \Z^n)<  \rho L \sqrt{ \log_+ \frac{\|U^{\mathrm{T}} \boldsymbol{\theta}\|}{L}} \right \}
	$$
	\item For a complex vector $\Bz = \Bx + i \By \in \C^n$ with $\Bx, \By \in \R^n$, we define the LCD of $\Bz$ to be the LCD of the matrix 
	$$
	\left(\begin{array}{c}
	\Bx^{\mathrm{T}} \\
	\By^{\mathrm{T}}
	\end{array} \right).
	$$
\end{itemize}
$\rho$ is a parameter that is not normally included in the definition, but we will need this extra flexibility in the appendix when we handle directed adjacency matrices, which does not have iid entries.  For any fixed $\rho$, the only effect is to slightly alter the constants in the following theorems.  For the remainder of the paper we set $\rho =1$ for convenience and only utilize this general $\rho$ in Section \ref{sec:directedrandomgraphs}.
\end{definition}

Our first lemma shows that the LCD of a complex vector is invariant under rotations by a complex phase.
\begin{lemma} \label{lem:phaseLCD}
	For $\Bz \in S^{n-1}_{\C}$, $D(\Bz) = D(e^{i \phi} \Bz)$ for any $\phi \in \R$.  
\end{lemma}
\begin{proof}
	Let 
	\begin{equation}\label{eq:rotation}
	R(\phi) = \left(\begin{array}{cc} 
	\cos(\phi) & \sin(\phi) \\
	-\sin(\phi) & \cos(\phi)
	\end{array} \right).
	\end{equation}
	Note that 
	\begin{align*}
	D(z) &= \inf \left \{ \|\boldsymbol{\theta}\|: \boldsymbol{\theta} \in \R^2, \dist(U^{\mathrm{T}} \boldsymbol{\theta}, \Z^n)<  L \sqrt{ \log_+ \frac{\|U^{\mathrm{T}} \boldsymbol{\theta}\|}{L}} \right \} \\
	&= \inf \left \{ \|R(\phi)\boldsymbol{\theta}\|: \boldsymbol{\theta} \in \R^2, \dist(U^{\mathrm{T}} R(\phi) \boldsymbol{\theta}, \Z^n)<  L \sqrt{ \log_+ \frac{\|U^{\mathrm{T}} R(\phi)\boldsymbol{\theta}\|}{L}} \right \} \\
	&= D(e^{i \phi} \Bz).
	\end{align*}
\end{proof}

The next lemma shows that one can always rotate a complex vector so that the LCD of $\Bz$ is exhibited by the real component of the rotated vector.  
\begin{lemma}\label{lem:LCDrealpart}
	For $\Bz  = \Bx + i \By$, there exists $\phi \in [0, 2\pi]$ such that for $e^{i \phi} \Bz = \Bx' + i \By'$,
	$$D(\Bz) = D(\Bx').$$  
\end{lemma}
\begin{proof}
	Recall from Definition \ref{def:LCD},
	\begin{align*}
	D(\Bz) &= \inf \left \{ \|\boldsymbol{\theta}\|: \boldsymbol{\theta} \in \R^2, \dist(U^{\mathrm{T}} \boldsymbol{\theta}, \Z^n)< L \sqrt{ \log_+ \frac{\|U^{\mathrm{T}} \boldsymbol{\theta}\|}{L}} \right \}
	\end{align*}
	where $U = \left(\begin{array}{c}
	\Bx^{\mathrm{T}} \\
	\By^{\mathrm{T}}
	\end{array} \right).$
	Let $\theta \in \R^2$ such that $\dist(U^{\mathrm{T}} \boldsymbol{\theta}, \Z^n)<  L \sqrt{ \log_+ \frac{\|U^{\mathrm{T}} \boldsymbol{\theta}\|}{L}}$.  For any such $\boldsymbol{\theta}$, there exists a $\phi$ such that $R(\phi) \boldsymbol{\theta} = \left(\begin{array}{c}
	1 \\
	0
	\end{array} \right)$ where $R(\phi)$ is defined in (\ref{eq:rotation}).  This implies that $\dist(\|\boldsymbol{\theta}\| \re(e^{i \phi} \Bz)) < L \sqrt{ \log_+ \frac{\|\boldsymbol{\theta}\| \re(e^{i \phi} \Bz)}{L}}$ which corresponds to the defining relation for the LCD of the real part of $e^{i \phi} \Bz$.  By Lemma \ref{lem:phaseLCD} and its proof, taking the infimum proves the result.
\end{proof}

The crucial relationship between structure and small-ball probability is quantified in the next theorem.  
\begin{theorem}[\cite{RVnogaps}] \label{thm:smallballcomplex}
	Consider a random vector $\boldsymbol{\xi} = (\xi_1, \dots, \xi_n)$ where $\xi_i$ are i.i.d. copies of a real random variable $\xi$ that satisfy Assumption \ref{assump:main}.  Let $U \in \R^{m \times n}$ be fixed.  Then for every $L \geq \sqrt{\frac{8 m}{ q}}$ (where $q$ is the parameter from Assumption \ref{assump:main}) and $t \geq 0$, we have
	$$
	\sup_{\Bx \in \R^m} \P(\|U \boldsymbol{\xi} - \Bx\|_2 \leq t \sqrt{m}) \leq \frac{(C_{\ref{thm:smallballcomplex}}L/\sqrt{m})^m}{\det(UU^{\mathrm{T}})^{1/2}} \left(t + \frac{\sqrt{m}}{D(U)} \right)^m
	$$
\end{theorem}

We fix a constant $L \geq \sqrt{16/q}$ for the remainder of the proof since we will only apply the above theorem for $m \leq 2$. 
A simple argument shows that if we restrict our attention to incompressible vectors, the smallest value the LCD can take is on the order of $\sqrt{n}$.  
\begin{lemma} [\cite{rudelson2008LO, vershynin2014symmetric}] \label{lem:LCDincomp}
	Let $\Bx \in S^{n-1}_{\R}$ in $\Incomp_{\R}(\mathfrak{a}, \mathfrak{b})$, then there exists a constant $c_{\ref{lem:LCDincomp}} > 0$, such that
	$$
	D(\Bx) \geq c_{\ref{lem:LCDincomp}} \sqrt{n}.
	$$
\end{lemma}	

\subsection{Small ball Probabilities depending on real-imaginary correlations}
We adapt the notions of LCD to handle complex vectors.  This section follows previous developments in this direction \cite{RVnogaps, luh2018complex, ge2017eigenvalue}.
\begin{definition} \label{def:complexvector}
	For $\Bz = \Bx + i \By \in S^{n-1}_{\C}$, we let $\tilde{\Bz} = \left( \begin{array}{c} \Bx \\
	\By \end{array} \right)$ and
	$$
	V = V(\Bz) := \left(\begin{array}{c} \Bx^{\mathrm{T}} \\
	\By^{\mathrm{T}}  \end{array} \right) \in \R^{2 \times n}.	$$
	We define the real-imaginary correlation of $z$ to be
	$$
	d(\Bz) = \det(V V^{\mathrm{T}})^{1/2} = \sqrt{ \|\Bx\|_2^2 \|\By\|_2^2 - (\Bx \cdot \By)^2}.
	$$
\end{definition}

\begin{lemma} \label{lem:lowercorrelation}
	If $\Bz \in \Incomp_{\C}(a,b)$ with $\|M \Bz\| \leq c_{\ref{lem:lowerboundrandi}} \delta \sqrt{n}$ then there exists a constant $c_{\ref{lem:lowercorrelation}} > 0$ such that 
	$$
	d(\Bz) \geq c_{\ref{lem:lowercorrelation}} \delta.
	$$
\end{lemma}
\begin{proof}
	We first prove the claim that
	$$
	\min_{\theta \in \R} \| \Re(e^{i \theta} \Bz)\|^2 = \frac{1}{2} - \frac{\sqrt{1 - 4 d(\Bz)^2}}{2}.
	$$
	Since
	$$
	\Re(e^{i \theta} \Bz)= \cos(\theta) \Bx - \sin(\theta) \By,
	$$
	the extremal values of $\Re(e^{i\theta} \Bz)$ are the singular values of the matrix $(\Bx \, \By)$.  These can be calculated from the eigenvalues of 
	$$
	\left( \begin{array}{c} \Bx^{\mathrm{T}} \\ \By^{\mathrm{T}} \end{array} \right)(\Bx \, \, \By) = 
	\left(\begin{array}{cc} \|\Bx\|^2 & x\cdot \By \\
	x \cdot \By & \|\By\|^2 \end{array} \right)
	$$  
	which are the solutions of 
	$$
	\lambda^2 - \lambda + d(\Bz)^2 = 0.  
	$$
	Solving the quadratic equation and choosing the larger root yields the claim.  
	
	By Lemma \ref{lem:lowerboundrandi}, the real part of any vector that satisfies the requirements of the lemma has norm bounded below by $c_{\ref{lem:lowerboundrandi}} \delta$, so we must have that
	$$
	\frac{1}{2} - \frac{\sqrt{1 - 4 d(\Bz)^2}}{2} \geq c_{\ref{lem:lowerboundrandi}}^2 \delta^2.
	$$
	Simplifying this inequality gives
	$$
	d(\Bz)^2 \geq (c_{\ref{lem:lowerboundrandi}} \delta)^2 - (c_{\ref{lem:lowerboundrandi}} \delta)^4 \geq c^2 \delta^2.
	$$
	for a small enough constant $c$.  
\end{proof}

In the remainder of this section, we use a covering argument to exclude vectors with small LCD.  For real matrices, this type of argument appeared in \cite{rudelson2008LO}.  However, the main difficulty in the current setting, is that we must consider complex spheres, which have dimension $2(n-1)$ when embedded into the real Euclidean space.  On the other hand, we are left with the same amount of randomness as in the real case.  To handle this difficulty, we divide the remaining vectors into two classes, \emph{genuinely complex} and \emph{essentially real}.  For genuinely complex vectors, the small-ball probabilities are greatly improved as the real and imaginary components are uncorrelated.  This is enough to compensate for the added dimensionality.  Essentially real vectors have highly correlated real and imaginary parts and so can be thought of as residing in a lower-dimensional space.  For this class of vectors, a variant of the original covering argument from \cite{rudelson2008LO} suffices.  This two-class approach is due to \cite{RVnogaps} and has been expanded upon in \cite{ge2017eigenvalue}.        
\begin{definition} \label{def:complexreal}
	Fix a scale $\alpha \in [c_{\ref{lem:lowerboundrandi}} \delta, 1]$ for $\|\Bx\|$.  Take $D \in [c_{\ref{lem:LCDincomp}} \sqrt{n}/ \alpha, D_0]$ for the LCD where
	\begin{equation} \label{eq:D0}
	D_0 = e^{c_{\star} n} \leq L e^{\mu^2 n/ L^2}
	\end{equation}
	where we lower the value of $c_{\star} > 0$ if necessary
	 and let $d_0 =  \max\left( \frac{L}{D} \sqrt{\log_+ \frac{D \alpha}{L}}, \frac{\sqrt{n} \alpha}{D} \right)$.
	\begin{itemize}
		\item (Genuinely Complex $\Bz$) For $d_0 \leq d \leq 1$, define
		\begin{align} \label{eq:complex}
		S_{D, d, \alpha} := \Big\{\Bz = \Bx+ i\By &\in \Incomp(a, b): \alpha \leq \|\Bx\| \leq 2 \alpha, D(\Bz)  = D(\Bx), \\
		& \frac{\Bx}{\|\Bx\|} \in \Incomp_{\R}(\mathfrak{a}, \mathfrak{b}), D \leq D(\Bz) \leq 2 D, d \leq d(\Bz) \leq 2 d\Big\}. \nonumber
		\end{align}
		
		\item (Essentially real $\Bz$)
		Define
		\begin{align} \label{eq:real}
		S_{D, d_0, \alpha} := \Big\{\Bz = \Bx+ i\By &\in \Incomp(a, b): \alpha \leq \|\Bx\| \leq 2 \alpha, D(\Bz)  = D(\Bx), \\
		& \frac{\Bx}{\|\Bx\|} \in \Incomp_{\R}(\mathfrak{a}, \mathfrak{b}), D \leq D(\Bz) \leq 2 D, d(\Bz) \leq d_0 \Big\}. \nonumber
		\end{align}
	\end{itemize}
\end{definition}

The next proposition establishes a strong small-ball probability for genuinely complex vectors.
\begin{proposition} \label{prop:tensorcomplex}
	For $\Bz \in S_{D, d, \alpha}$,  $L \geq \sqrt{16/p}$ and $t \geq 0$, 
	$$
	\P(\| M \Bz\| \leq t \sqrt{n}) \leq \left( \frac{C_{\ref{prop:tensorcomplex}} L^2}{d} \left(t + \frac{1}{D} \right)^2 \right)^{n}
	$$
\end{proposition}
\begin{proof}
	Let $M_j$ denote the $j$-th row of $M$. We have that
	$$
	|M_j \Bz| = |N_j \Bx + i N_j \By + \lambda z_j| = \left \| V N_j^{\mathrm{T}}  - \left( \begin{array}{c} \re(\lambda z_j) \\ \im(\lambda z_j)\end{array}\right)\right \|
	$$
	where we recall from Definition \ref{def:complexvector} that $V = V(\Bz) = \left( \begin{array}{c} \Bx^{\mathrm{T}} \\
	\By^{\mathrm{T}} \end{array} \right)$ and $N_j$ is the $j$-th column of $N$, where by assumption each entry is i.i.d.  
	Specializing Theorem \ref{thm:smallballcomplex} to our setting we arrive at 
	\begin{align*}
	\P(|M_j \Bz| \leq t \sqrt{2}) &\leq \frac{(C_{\ref{thm:smallballcomplex}}L/\sqrt{2})^2}{\det(V V^{\mathrm{T}})^{1/2}} \left(t + \frac{\sqrt{2}}{D} \right)^2 \\
	&\leq \frac{C}{d(\Bz)} \left(t + \frac{\sqrt{2}}{D} \right)^2
	\end{align*}
	where in the last line we utilized the observation that $\det(V V^{\mathrm{T}}) = \|\Bx\|^2 \|\By\|^2 - (\Bx \cdot \By)^2 = d(\Bz)^2$. A quick change of variables from $\sqrt{2} t$ to $t$ puts the single coordinate bound into the desired form.  To extend this bound to the entire vector, we use a standard tensorization argument, which completes the proof. 
\end{proof}

The following proposition is proved analogously and is again a simple consequence of tensorization and our definition of $d_0$.
\begin{proposition} \label{prop:tensorreal}
	For $\Bz \in S_{D, d_0, \alpha}$, $L \geq \sqrt{16/p}$ and $t \geq 0$,
	$$
	\P(\|N \Bx - \delta \sqrt{ n} \By\| \leq t \sqrt{n}) \leq \left( \frac{C_{\ref{prop:tensorreal}}}{\alpha} \left( t  + \frac{1}{D} \right) \right)^{n}.
	$$
\end{proposition}

\subsection{Nets} \label{subsec:nets}  In this section, we construct discrete nets of various level sets partitioned by real-complex correlation and LCD.  
\subsubsection{Genuinely Complex Case}

\begin{proposition} \label{prop:complexnet}
	Recall the definition of $S_{D, d, \alpha}$ from (\ref{eq:complex}).  For any fixed constant $\mu >0$, there exists a $\frac{\mu \sqrt{n}}{D}$-net of $S_{D, d, \alpha}$ with cardinality bounded by
	$$
	\frac{C_{\ref{prop:complexnet}}^{2n} D^{2n+1} d^{n-1}}{\mu^{n+1} n^{n + 1/2}}
	$$
	where $C_{\ref{prop:complexnet}}$ is an absolute constant.
\end{proposition}

\begin{proof}
	By the definition of LCD, there exists a $\Bp \in \Z^n$ such that
	$$
	\|D(\Bz) - \Bp\| <  L \sqrt{ \log_+ \frac{\|D(\Bz) x\|}{L}}.
	$$
	Therefore,
	\begin{align*}
	\|\Bp\| &\leq \|D(\Bz) \Bx\| +  L \sqrt{ \log_+ \frac{\|D(\Bz)  \Bx\|}{L}} \\
	&\leq D \alpha +  \left(\frac{L}{\|D(\Bz)  \Bx\|} \sqrt{ \log_+ \frac{\|D(\Bz)  \Bx\|}{L}}  \right) \|D(\Bz)  \Bx\| \\
	&\leq C D \alpha
	\end{align*}
	for some univversal constant $C > 0$ where in the last line we have used that the function $\frac{1}{x} \sqrt{ \log_+ x}$ is bounded.
	Using the triangle inequality in the other direction gives
	$$
	\|\Bp\| \geq c D \alpha
	$$
	for a small universal constant $c > 0$.
	By definition, 
	$$
	d(\Bz) = s_1(V) s_2(V).
	$$
	Since $\Bz$ is a unit vector, at least one of $\Bx$ or $\By$ has norm greater than $1/\sqrt{2}$.  Since,
	$$
	\left \|V^{\mathrm{T}} \frac{\Bx}{\|\Bx\|}\right\| \geq \|\Bx\| \text{ and } \left\|V^{\mathrm{T}} \frac{\By}{\|\By\|}\right\| \geq \|\By\|,
	$$
	we must have that $s_1(V) \geq 1/\sqrt{2}$ which implies that 
	$$
	s_2(V) \leq 2 \sqrt{2} d.	
	$$
	Define 
	$$
	W = \left(\begin{array}{cc} \Bp^{\mathrm{T}} \\
	D(\Bz) \By^{\mathrm{T}} \end{array}\right).
	$$
	By definition,
	$$
	\|D(\Bz) V - W\| = \left \| \left(\begin{array}{c} D(\Bz) \Bx^{\mathrm{T}} - \Bp\\
	0 \end{array}\right) \right\| \leq  L \sqrt{ \log_{+} \frac{D(\Bz) \|\Bx\|}{L}} .
	$$
	From Weyl's inequality we can deduce that
	$$
	s_2(W) \leq s_2(D(\Bz) V) +  L \sqrt{ \log_{+} \frac{D(\Bz) \|\Bx\|}{L}} \leq 4 \sqrt{2} D d + L \sqrt{ \log_{+} \frac{D(\Bz) \|\Bx\|}{L}} .
	$$
	We write $\det(WW^{\mathrm{T}})^{1/2}$ in two ways via the product of singular values and the volume of a parallel piped.  In particular, 
	$$
	\|\Bp\| \cdot \|P_{\Bp^{\perp}} D(\Bz) \By\| = s_1(W) s_2(W)
	$$
	where $P_{\Bp^\perp}$ is the operator that projects onto the subspace orthogonal to $\Bp$.
	Since $s_1(W) \leq \|\Bp\| + \|D(\Bz) \By)\|$, 
	$$
	\|P_{\Bp^\perp} D(\Bz) \By\| \leq \left( 1 + \frac{\|D(\Bz) \By\|}{\|\Bp\|} \right) s_2(W).
	$$
	Recalling that $\|\Bp\| \geq c D \alpha$ and $\|D(\Bz) \By\| \leq 2 D$, we find that
	$$
	\|P_{\Bp^{\perp}} D(\Bz) \By\|_2 \leq \left(1  + \frac{4}{c \alpha}\right) \left(4 \sqrt{2} D d +  L \sqrt{ \log_{+} \frac{D(\Bz) \|\Bx\|}{L}}  \right) \leq C' \left( \frac{D d}{\alpha} + \frac{ L}{\alpha} \sqrt{ \log_{+} \frac{D(\Bz) \|\Bx\|}{L}}  \right)
	$$
	for another universal constant $C' > 0$.
	As we are in the genuinely complex case, 
	$$
	d \geq \frac{L}{D} \sqrt{ \log_{+} \frac{D(\Bz) \|\Bx\|}{L}},  
	$$
	so
	\begin{equation} \label{eq:pperp}
	\|P_{\Bp^{\perp}} D(\Bz) \By\| \leq C'' \frac{D d}{\alpha}.
	\end{equation}
	We now have the estimates to construct a $\frac{\mu \sqrt{n}}{D}$-net of $S_{D, d, \alpha}$.  For any $\Bx + i \By \in S_{D, d, \alpha}$ there exists $\Bp \in \Z^n \cap B(0, C D \alpha)$ such that
	$$
	\left \| \Bx - \frac{\Bp}{D(\Bz)} \right \| < \frac{L}{D(\Bz)} \sqrt{ \log_{+} \frac{D(\Bz) \|\Bx\|}{L}} \leq \frac{\mu \sqrt{n}}{D}
	$$
	where the last inequality follows from $D_0 \leq L e^{\mu^2 n/ L^2}$ from (\ref{eq:D0}).  We work with at most  $CD/\mu \sqrt{n}$ discrete multiples of $\Bp$ that approximate $\Bp/D(\Bz)$ up to an accuracy of $\mu \sqrt{n}/D$.  Therefore, to bound the number of discrete multiples we have to consider, we multiply the number of lattice points in $B(0, C D \alpha)$ by the number of discrete multiples to get a bound of 
	\begin{equation} \label{eq:pmultiples}
	\frac{CD}{\mu \sqrt{n}} \left( \frac{CD \alpha}{\sqrt{n}} \right)^n.
	\end{equation}
	For each discrete scaling of a lattice point $\Bp$, we have by (\ref{eq:pperp}) that
	$$
	\|P_{\Bp^{\perp}} \By\|_2 \leq\frac{C'' d}{\alpha}
	$$
	so $\By$ must lie in a cylinder of radius $C''d/\alpha$ in the direction of $\Bp$.  This crucial observation severely restricts the space of potential $\By$.
	 Using the standard volume argument gives a $\mu \sqrt{n}/D$-net of this cylinder with size bounded by 
	$$
	\frac{CD}{\mu \sqrt{n}} \left(\frac{C'' D d}{\mu \sqrt{n} \alpha} \right)^{n-1}.
	$$
	Combining these bounds yields the result since $d/\alpha \geq \sqrt{n}/D$ by the assumption that $d_0 \geq \sqrt{n} \alpha/ D$.
\end{proof}

\subsubsection{Essentially Real Case}
 
\begin{proposition} \label{prop:realnet}
	For any constant $\mu > 0$, there exists a set $\mathcal{N}$ with cardinality bounded by 
	$$
	\frac{C_{\ref{prop:realnet}}^{2n+1} D^{n+2} \alpha^n}{\mu^2 \sqrt{n}^{n+2}}
	$$
	such that for every $\Bz = \Bx + i \By \in S_{D, d_0, \alpha}$ there is $\Bu + i \Bv \in \mathcal{N}$ such that $\|\Bx- \Bu\| \leq  \frac{\mu \sqrt{n}}{D}$ and $\|\By - \Bv\| \leq \frac{\mu \sqrt{n}}{D \alpha}$. 
\end{proposition}

\begin{proof}
	We begin with the case where $d(\Bz) <  \frac{ L}{D} \sqrt{ \log_{+} \frac{D \alpha}{L}}$.  We can recycle many of the estimates from the genuinely complex case.  However, the estimate for the projection of $\By$ onto the subspace orthogonal to $\Bp$ changes.  Now we have
	$$
	\|P_{\Bp^\perp} \By\|_2 \leq \frac{ L}{D \alpha} \sqrt{ \log_+ \frac{D \alpha}{L}} \leq \frac{\mu \sqrt{n}}{D \alpha}
	$$
	where the last inequality follows from choosing $c_{\star}$ small enough in the definition of $D_0$ in (\ref{eq:D0}).  Again, using the discrete multiples of lattice points to approximate $\Bx$ with cardinality bounded by (\ref{eq:pmultiples}).  For each discrete multiple of a lattice point, we can match it with a net of size 
	$$
	\frac{C D \alpha}{\mu \sqrt{n}} \left(\frac{C'' D d}{\mu \sqrt{n} \alpha} \right)^{n-1} \leq 	\frac{C D \alpha}{\mu \sqrt{n}}  C^{n-1}
	$$
	where the inequality follows from our bound on $d_0$.
	Therefore, the total net size is bounded by
	$$
	\frac{C D}{\mu \sqrt{n}} \left(\frac{C D \alpha}{\sqrt{n}} \right)^n \frac{C D \alpha}{\mu \sqrt{n}} C^{n-1}.
	$$
	Finally, we address the case where $d(\Bz) \geq  \frac{L}{D} \sqrt{ \log_+ \frac{D \alpha}{L}}$ and $d(\Bz) <  \frac{\sqrt{n} \alpha}{D}$.  In this, case we use the bound
	$$
	\frac{CD}{\mu \sqrt{n}} C^{n-1}
	$$
	to control the number of $\By$ contained in the cylinder and proceed as in the genuinely complex case to obtain a net of size less than
	$$
	\frac{C D}{\mu \sqrt{n}} \left(\frac{C D \alpha}{\sqrt{n}} \right)^n \frac{C D }{\mu \sqrt{n}} C^{n-1}.
	$$  
\end{proof}

\subsection{Completing the Proof of the Structure Theorem}
\subsubsection{Genuinely Complex Case}
\begin{theorem} \label{thm:levelcomplex}
	$$
	\P\left(\inf_{\Bz \in S_{D, d, \alpha}} \|M \Bz\| \leq \frac{K \mu n}{D} \text{ and }  \EE_K \right) \leq e^{-c_{\ref{thm:levelcomplex}} n}.
	$$
\end{theorem}

\begin{proof}
	As shown in Proposition \ref{prop:complexnet}, there exists a $\mu \sqrt{n}/D$-net, $\mathcal{N}$ of $S_{D, d, \alpha}$ with cardinality at most
	$$
	\frac{C_{\ref{prop:complexnet}}^{2n} D^{2n+1} d^{n-1} \alpha}{\mu^{n+1} \sqrt{n}^{2n+1}}.
	$$
	Suppose that $\Bz' \in \mathcal{N}$ and $\|\Bz- \Bz'\| \leq \mu \sqrt{n}/D$.  Then the event $\|M \Bz\| \leq  \frac{K \mu n}{D}$ implies that
	$$
	\|M \Bz'\| \leq \|M\| \|\Bz - \Bz'\| + \frac{K \mu n}{D} \leq \frac{2 K \mu n}{D}.
	$$
	Taking $t = 2 K \mu \sqrt{n} /D$, we have by Proposition \ref{prop:tensorcomplex}, 
	\begin{align*}
	\P(\inf_{\Bz \in S_{D, d, \alpha}} \|M \Bz\| \leq K \mu n/D \text{ and } \EE_K) &\leq  \P(\inf_{\Bz' \in \mathcal{N}} \|M \Bz'\| \leq t \sqrt{n}) \\
	&\leq \frac{C_{\ref{prop:complexnet}}^{2n} D^{2n+1} d^{n-1}}{\mu^{n+1} n^{n + 1/2}} \frac{2^n C_{\ref{prop:tensorcomplex}}^{n} L^{2n} (2 K \mu \sqrt{n})^{2 n}}{d^{n} D^{2 n}}  \\
	&\leq C^n \mu^{n-1} \\
	&\leq e^{-cn}
	\end{align*}
	where the second to last inequality follows from the bound on $D_0$ in (\ref{eq:D0}) and the last line follows from choosing $\mu$ small enough.
	
\end{proof}

\subsubsection{Essentially Real Case}
\begin{theorem} \label{thm:levelreal}
	$$
	\P(\inf_{\Bz \in S_{D, d_0, \alpha}} \|M \Bz\| \leq K \mu n/D \text{ and } \|M\| \leq K \sqrt{n}) \leq e^{- c_{\ref{thm:levelreal}}n }.
	$$
\end{theorem}
\begin{proof} 

	Suppose that we are in the event that for some $\Bz = \Bx+ i\By \in S_{D, d_0, \alpha}$, 
	$$
	\|M \Bz\| \leq \mu n/D.  
	$$
	The real part of the inequality gives
	$$
	\|N \Bx- \delta \sqrt{n} \By\| \leq \mu n/D.  
	$$
	By Proposition \ref{prop:realnet}, there exists a net $\mathcal{N}$ with cardinality bounded by
	$$
	\frac{C_{\ref{prop:realnet}}^{2n+1} D^{n+2} \alpha^n}{\mu^2 \sqrt{n}^{n+2}}
	$$
	such that there is a $\Bu + i\Bv \in \mathcal{N}$ with $\|\Bx-\Bu\| \leq  \frac{\mu \sqrt{n}}{D}$ and $\|\By- \Bv\| \leq \frac{\mu \sqrt{n}}{D \alpha}$.  
	We therefore have
	\begin{align*}
	\|N \Bu - \delta \sqrt{n} \Bv\| &\leq \|N \Bu - N \Bx\| + \|N \Bx - \delta \sqrt{n} \By\| + \|\delta \sqrt{n} \By - \delta \sqrt{n} \Bv\|  \\
	&\leq K \sqrt{n} \frac{ \mu \sqrt{n}}{D} + K \mu n/D + \delta \sqrt{n} \frac{\mu \sqrt{n}}{D \alpha} \\
	&\leq C \mu n/D.
	\end{align*}
	In the last line, we used the fact that $\alpha \geq c \delta$.  We therefore have that for $t = K \mu \sqrt{n}/D$ in Proposition \ref{prop:tensorreal}, 
	\begin{align*}
		\P(\inf_{\Bz \in S_{D,d_0, \alpha}} \|M \Bz\| \leq K \mu n/D) &\leq \P(\inf_{\Bu + i\Bv \in \mathcal{N}} \|N \Bu - \delta \sqrt{n} \Bv\| \leq \frac{\nu \sqrt{n} \sqrt{n}}{D}) \\
		&\leq \frac{C_{\ref{prop:realnet}}^{2n+1} D^{n+2} \alpha^n}{\mu^2 \sqrt{n}^{n+2}} \left( \frac{C_{\ref{prop:tensorreal}}}{\alpha} \left( \frac{2 K \mu \sqrt{n}}{D}\right) \right)^{n} \\
		&\leq \frac{C^n D^2}{n} \mu^{n-2} \\
		&\leq e^{-c n}
	\end{align*}
	where in the last line we used the bound $D \leq D_0$ and chose $\mu$ small enough. 
\end{proof}

\subsubsection{Combining all the elements}
In this section we aggregate all the previous results to deduce that near-null vectors must have large LCD.

\begin{proof}[Proof of Theorem \ref{thm:nullvectors}]
	By Lemma \ref{lem:comp} and the observation that $ K \mu n/ \widetilde{D} \leq c_{\ref{lem:comp}} \sqrt{n}$, the event 
	$$
	\inf_{\Bz \in \Comp(a,b)} \|M \Bz\|_2 \leq K \mu n/ \widetilde{D}  \text{ and } \EE_K 
	$$
	occurs with probability at most $e^{-c_{\ref{lem:comp}}n}$.  Next, we exclude those vectors with compressible real part.  Note that $K \mu n/ \widetilde{D} \leq c_{\ref{lem:lowerboundrandi}} \delta \sqrt{n}$ by our lower bound on $\widetilde{D}$ and decreasing $\mu$ if necessary.  Thus by Lemma \ref{lem:lowerboundrandi}, on the event that $ \|M \Bz\| \leq K \mu n/ \widetilde{D}$, we need only consider complex unit vectors $\Bz = \Bx + i\By \in S_{\C}^{n-1}$  with $\|\Bx\| \geq c_{\ref{lem:lowerboundrandi}} \delta$.  Let 
	$$
	T = \{\Bz=\Bx+i\By \in \Incomp(a,b): c_{\ref{lem:lowerboundrandi}} \delta \leq \|\Bx\| \leq 1, \Bx/\|\Bx\| \in \Comp_{\R}(\a, \b) \}.
	$$
	Choosing dyadic points $\alpha_k = 2^{k} c_{\ref{lem:lowerboundrandi}} \delta$  in $[c_{\ref{lem:lowerboundrandi}} \delta, 1]$ for $k \in \N$, we can take a union bound to conclude that
	\begin{align*}
	\P\left(\inf_{\Bz \in T} \|M \Bz\| \leq K \mu n/ \widetilde{D}  \text{ and } \EE_K \right) &\leq \sum_{k} \P\left(\inf_{\Bz \in T_{\alpha_k}} \|M \Bz\| \leq K \mu n/ \widetilde{D} \text{ and } \EE_K \right) \\
	&\leq \sum_k e^{-c_{\ref{lem:realcomp}} n} \\
	&\leq e^{- cn}
	\end{align*}
	where in the second inequality we invoked Lemma \ref{lem:realcomp} and in the last line we noted that the number of non-zero summands is bounded by $n$ from the lower bound on $\delta$.  We direct our attention to vectors with incompressible real part.  By Lemmas \ref{lem:phaseLCD} and \ref{lem:LCDrealpart}, it suffices to consider vectors whose LCD's are attained by their real component, or in other words $\Bz = \Bx+i\By$ such that $D(\Bz) = D(\Bx)$.  Now, we gradually exclude level sets by LCD, norm of the real component, and real-imaginary correlation. By Lemma \ref{lem:lowerboundrandi}, Lemma \ref{lem:LCDincomp} and Lemma \ref{lem:lowercorrelation}, we need only consider vectors $\Bz = \Bx+i\By$ such that $\|\Bx\| \geq c_{\ref{lem:lowerboundrandi}} \delta$, $D(\Bz) \geq c_{\ref{lem:LCDincomp}} \sqrt{n}$ and $d(\Bz) \geq c_{\ref{lem:lowercorrelation}} \delta$. We define
	$$
	D_h = 2^h c_{\ref{lem:LCDincomp}} \sqrt{n} \text{ and } d_j = 2^j d_0.
	$$
	Then we denote
	$$
	S_{complex} = \bigcup_{h, j, k \in \N : D_h \leq \widetilde{D}/2} \left\{z \in S_{\C}^{n-1}: S_{D_h, d_j, \alpha_k} \right\} 
	$$ 
	and 
	$$
	S_{real} = \bigcup_{h, k \in \N: D_h \leq \widetilde{D}/2 } \left\{z \in S_{\C}^{n-1}: S_{D_h, d_0, \alpha_k} \right\}. 
	$$
	Then by Theorems \ref{thm:levelcomplex} and \ref{thm:levelreal},
	\begin{align*}
	\P \Bigg( \inf_{\Bz \in S_{complex} \cup S_{real}} &\|M \Bz\| \leq K \mu n/ \widetilde{D} \text{ and } \EE_K \Bigg) \\
	&\leq \sum_{h, j, k \in \N : D_h \leq \widetilde{D}/2} \P \left(\inf_{\Bz \in S_{D_h, d_j, \alpha_k}} \|M \Bz\| \leq K \mu n/ D_h \right) \\
	&\qquad  + \sum_{h, k \in \N : D_h \leq \widetilde{D}/2} \P \left(\inf_{\Bz \in S_{D_h, d_0, \alpha_k}} \|M \Bz\| \leq K \mu n/ D_h \right) \\
	&\leq \sum_{h, j, k \in \N : D_h \leq \widetilde{D}/2} e^{-c_{\ref{thm:levelcomplex}} n} + \sum_{h, k \in \N : D_h \leq \widetilde{D}/2} e^{-c_{\ref{thm:levelreal}} n} \\
	&\leq e^{- cn}
	\end{align*} 
	for a small enough constant $c < 0$.  Combining all the error terms completes the proof.  
\end{proof}

\begin{definition}[L\'evy concentration]
	Let $\boldsymbol{\xi}$ be a random vector whose entries are iid copies of a random variable that satisfies Assumption \ref{assump:main}.  
	For a complex vector $\Bz$, we define the L\'evy concentration of $\Bz$ to be
	$$
	\rho(\Bz, t) = \sup_{r \in \C} \P(|\boldsymbol{\xi} \cdot \Bz - r| \leq t), 
	$$
	where $\boldsymbol{\xi} \cdot \Bz = \boldsymbol{\xi}^{\mathrm{T}} \Bz$ is the dot product of $\boldsymbol{\xi}$ and $\Bz$.  
\end{definition}

Finally, we quote a well-known reslult for our random matrix shifted by a \emph{real} value.
\begin{theorem} [\cite{rudelson2008LO, luh2018eigenvector}]\label{thm:realshift}
	Let $\lambda \in \R$ with $|\lambda | \leq K \sqrt{n}$.  
	With probability at least $1 - e^{- c_{\ref{thm:realshift}} n}$, any vector $\Bz \in S_{\C}^{n-1}$ with $\|M \Bz\| \ll \sqrt{n}/D$ is such that $c_{\ref{thm:realshift}} \sqrt{n} \leq D(\Bz) \leq e^{c_{\ref{thm:realshift} } n}$ so that for $t \geq 0$,
	$$
	\rho(\Bz, t) \leq C_{\ref{thm:realshift}}  \left(t + \frac{1}{D} \right).
	$$
\end{theorem}

\begin{remark}
	The exact form of the above theorem does not appear in the literature but can easily be deduced from the proofs in \cite{rudelson2008LO, luh2018eigenvector}.
\end{remark}

\section{Structure of Eigenvectors} \label{sec:eigstruc}

We now have the tools to prove the following eigenvector structure theorem.
\begin{theorem}\label{thm:vecstructure}
	For an $n \times n$ random matrix $N$ with iid entries that satisfy Assummption \ref{assump:main}, there exist constants $c_{\ref{thm:vecstructure}}, C_{\ref{thm:vecstructure}} > 0$, such that with probability at least $1 - e^{-c_{\ref{thm:vecstructure}} n}$, for all eigenvectors $\Bv$ of $N$, we have  
	$$
	\rho(\Bv, t) \leq C_{\ref{thm:vecstructure}} t + e^{-c_{\ref{thm:vecstructure}} n}
	$$
	for $t \geq 0$.
\end{theorem}

	We begin with a technical preliminary result. 
\begin{theorem} \label{thm:eigstructure} For any $c_{\ref{thm:realshift}} \sqrt{n} \leq D \leq e^{c_\star n}$, 
	\begin{align*}
	\P\Bigg(\exists &\text{ an eigenvector } \Bv \text{ of } N \text{ and } t \geq 0 \\
	&\text{ such that } \rho(\Bv, t) \geq C_{\ref{thm:eigstructure}}\left(t + \frac{D}{\sqrt{n}} t^2 + \frac{1}{D} \right) \text{ and } \EE_K \Bigg) \leq e^{-c_{\ref{thm:eigstructure}} n}
	\end{align*}
\end{theorem}

\begin{proof}
	To extend our previous results to eigenvectors, it is natural to discretize the complex ball of radius $K \sqrt{n}$, since we are assuming the eigenvalues are bounded by $K \sqrt{n}$.  Any eigenvector will then be an approximate null vector for some complex number in the ball.  However, the difficulty is that our small-ball probability bound in Theorem \ref{thm:smallballcomplex} depends on the real-imaginary correlation of our shift $\lambda$, which in turn is lower bounded by the imaginary component of $\lambda$.  Therefore, our upper bound on the L\'evy probability of approximate null-vectors degrades significantly as we near the real line.  The first step of our strategy is to control the L\'evy probability of approximate null vectors with corresponding approximate eigenvalues near the real line by comparing them to approximate eigenvectors of real shifts and invoking Theorem \ref{thm:realshift}, which naturally has no dependence on the imaginary component.  Taking a fine enough net of the real line thereby proves our theorem for eigenvalues inside a neighborhood of the real line.  In the next step, we work on the ball with a strip around the real line excluded.  This gives us some control on the imaginary component of the eigenvalues and allows us to use the results from Section \ref{sec:approximatenull}.

	We proceed with the first step.  Let $\beta = \sqrt{n}/D$.  
	There exists a $\beta/2$-net, $\mathcal{N}$ of the real interval $[-K \sqrt{n}, K \sqrt{n}] \in \C$ with $|\mathcal{N}| \leq 10K/\beta$.  At every point in $\mathcal{N}$, we place a ball of radius $\beta$.  The union of these balls necessarily contains a $c \beta$ neighborhood of the real interval $[-K, K]$. On the event that there exists an eigenvalue, $\lambda$, within the strip with eigenvector $v$, there must exist a $\lambda' \in \mathcal{N}$ such that
	$$
	\|(N - \lambda') \Bv \| = \|(\lambda - \lambda') \Bv\| \leq \beta.
	$$
	Therefore, by Theorem \ref{thm:realshift}, 
	\begin{equation}\label{eq:smallballforrealshift}
	\rho(\Bv, t) \leq C_{\ref{thm:realshift}}  \left(t + \frac{1}{D}\right)
	\end{equation}
	 with probability at most 
	 $$|\mathcal{N}|e^{-\bar{c}_{\ref{thm:realshift}} n} \leq  \frac{10 K \sqrt{n}}{\beta} e^{-\bar{c}_{\ref{thm:realshift}} n} \leq e^{- c n}
	 $$ 
	 after reducing $c_{\star}$ if necessary.  It is worth pointing out that any reduction in $c_{\star}$ will alter the constant in the error probability of Theorem \ref{thm:nullvectors}, but there is no circular dependence of constants.      
	
	Now, let $S'$ denote the centered disk of radius $K \sqrt{n}$ after removing the strip of width $\beta$ around the real line.  There exists a $\beta$-net, $\mathcal{N}'$ of $S'$ of size at most $C n/ \beta^2$.  Again, for an eigenvalue $\lambda \in S'$, there exists a $\lambda' \in \mathcal{N}'$ such that
	$$
	\|(N -\lambda') \Bv\| \leq \beta.  
	$$ 
	Note that by our choice of $\beta$, $D$ will satisfy the requirements of $\widetilde{D}$ in Theorem \ref{thm:nullvectors}.
	Thus, by Theorem \ref{thm:nullvectors}, with probability at least $1 - |\mathcal{N}'| e^{-c_{\ref{thm:nullvectors}} n} \geq 1 - e^{-c n}$, 
	any eigenvector, $\Bv \in S'$ will have $D(\Bv) \geq D_0$ and $d(\Bv) \geq c \beta$, since the imaginary component of any element in $S'$ is bounded below by $c \beta$.   
	By applying Theorem \ref{thm:smallballcomplex}, we obtain that for such a vector $\Bv$,
	\begin{equation} \label{eq:smallballforS'}
	\rho(\Bv, t) \leq \frac{C}{\beta} \left( t+ \frac{1}{D} \right)^2.
	\end{equation}
	Combining (\ref{eq:smallballforrealshift}) and (\ref{eq:smallballforS'}) completes the proof.
\end{proof}

As stated, the above theorem applies to a single choice of $D$.  The previous proofs can be restructured to show that in fact the statement holds for the whole range of $D$ \emph{simultaneously}.  However, to preserve clarity, we simply deduce this as a corollary of the previous theorem.  
\begin{corollary} \label{cor:alld}
	There exist constants $C_{\ref{cor:alld}}, c_{\ref{cor:alld}} > 0$ such that
	\begin{align*}
	\P\Bigg(\exists &\text{ an eigenvector } \Bv \text{ of } N, D \in [c_{\ref{thm:realshift}} \sqrt{n}, e^{c_\star n}] \text{ and } t \geq 0 \\
	&\text{ such that } \rho(\Bv, t) \geq C_{\ref{cor:alld}}\left(t + \frac{D}{\sqrt{n}} t^2 + \frac{1}{D} \right) \text{ and } \EE_K \Bigg) \leq e^{-c_{\ref{cor:alld}} n}.
	\end{align*}
\end{corollary}

\begin{proof}
	Let $d_k := c_{\ref{thm:realshift}} \sqrt{n} 2^k$.  By applying Theorem \ref{thm:eigstructure} with $D = d_k$, we have that with probability at least $1 - e^{-c_{\ref{thm:eigstructure}} n}$, any eigenvector $\Bv$ of $N$ is such that 
	$$
	\rho(\Bv, t) \leq C_{\ref{thm:eigstructure}} \left(t + \frac{d_k}{\sqrt{n}} t^2 + \frac{1}{d_k} \right).
	$$
	On this event, for any $d_k \leq D < d_{k+1}$, 
	$$
	\rho(\Bv, t) \leq C_{\ref{thm:eigstructure}} \left(t + \frac{D}{2 \sqrt{n}} t^2 + \frac{2}{D} \right) \leq 2 C_{\ref{thm:eigstructure}} \left(t + \frac{D}{ \sqrt{n}} t^2 + \frac{1}{D} \right)
	$$
	which shows that to extend the event in Theorem \ref{thm:eigstructure} on $D = d_k$ to the entire interval $[d_k, d_{k+1})$ at the cost of a universal constant.  Therefore, to extend the result to the entire range $c_{\ref{thm:realshift}} \sqrt{n} \leq D \leq e^{c_\star n}$, we simply take a union bound over all $k \in \N$ with 
	$$c_{\ref{thm:realshift}} \sqrt{n} \leq D \leq e^{c_\star n}.$$
	The number of such $k$ is clearly bounded by $n$ so by the union bound, our event of interest holds with probability at least $1 - n e^{-c_{\ref{thm:eigstructure}} n}$.
\end{proof}
	For a fixed $D$, the bound 
	$$
	C_{\ref{thm:eigstructure}} \left(t + \frac{D}{ \sqrt{n}} t^2 + \frac{1}{D} \right)
	$$
	only yields a non-trivial bound on the scale
	$$
	\frac{1}{D} \ll t \ll \frac{n^{1/4}}{\sqrt{D}}.
	$$ 
	However, as we have an identical bound for all $D$ simultaneously, we can allow $D$ to vary with $t$ to combine these scales into a single bound.  

\begin{proof}[Proof of Theorem \ref{thm:vecstructure}]
	We let $D = \sqrt{n}/t$.  By Corollary \ref{cor:alld}, with probability at least $1 - e^{-c_{\ref{cor:alld}} n}$, for any eigenvector $\Bv$ of $N$, 
	$$
	\rho(\Bv, t) \leq C_{\ref{cor:alld}}\left(t + \frac{D}{\sqrt{n}} t^2 + \frac{1}{D} \right) \leq 3 C_{\ref{cor:alld}} t,
	$$ 
	where the last inequality follows from our choice of $D$.
	As we can only apply Corollary \ref{cor:alld} when $D \in [c_{\ref{thm:realshift}} \sqrt{n}, e^{c_\star n}]$, the above bound holds when $t \in [\sqrt{n} e^{-c_{\star} n},  c_{\ref{thm:realshift}}^{-1}]$.  The upper bound on $t$ can be ignored after choosing $C_{}$ large enough.
\end{proof}

There are a variety of simpler results depending on our choice of $t$ and $D$.
For example, setting $t = 0$ and $D = e^{c_{\star} n}$ in Theorem \ref{thm:eigstructure} yields the following notable consequence.  
\begin{corollary}\label{cor:rhozero}
	$$\P\left(\exists \text{ eigenvector } \Bv \text{ of } N \text{ such that } \rho(\Bv, 0) \geq e^{- c_{\ref{cor:rhozero}} n} \text{ and } \EE_K \right) \leq e^{-c_{\ref{thm:eigstructure}} n}$$
\end{corollary}

In fact, the proof of Theorem \ref{thm:eigstructure} yields a slightly more general theorem.  
\begin{theorem}
	For a $c_{\ref{thm:realshift}} \sqrt{n} \leq D \leq e^{c_\star n}$,
	we call $\Bv$ a $D$-\emph{approximate eigenvector} of $N$ if there exists $\lambda \in \C$ such that $\|(N- \lambda) \Bv\|_2 \leq c \sqrt{n}/D$.  Then,
	\begin{align*}
	\P\Bigg(\exists &\text{ a } D \text{-approximate eigenvector } \Bv \text{ of } N \text{ and } t \geq 0 \\
	&\text{ such that } \rho(v, t) \geq C_{\ref{thm:eigstructure}}\left(t + \frac{1}{D} +  \frac{t^2 D}{\sqrt{n}}  \right)  \text{ and } \EE_K \Bigg) \leq e^{-c_{\ref{thm:eigstructure}} n}.
	\end{align*}
\end{theorem}

\section{Directed Erd\H{o}s--R\'enyi Random Graphs} \label{sec:directedrandomgraphs}
For a directed graph $G = ([n], E)$ with vertex set $[n]$ and edge set $E$, we recall that the adjacency matrix $A$ is defined by 
$$
A_{i,j} := \begin{cases} 1, & \text{if } (i,j) \in E, \\
0, & \text{otherwise}.
\end{cases}
$$ 
We define the directed Erd\H{o}s--R\'enyi random graph to be the random graph on vertex set $[n]$ such that each edge $(i,j)$ appears independently with probability $p$, for a constant $p \in (0, 1)$.  For this model, the adjacency matrix is a random matrix with expectation $pJ$ or $p(J-I)$ where $J$ is the $n \times n$ matrix of ones depending on whether or not we exclude the possibility of loops. The extra $pI$ factor does not affect our arguments as it simply shifts the spectrum slightly.  The results from the previous section are not immediately relevant as this matrix model has large norm and so $\EE_K$ actually occurs with probability $o(1)$.  However, due to the low rank structure of $pJ$, we can extend the covering arguments to handle this case (cf. \cite{basak2017invertibility, luh2018sparse, lopatto2019tail}).   We let $\EE_K$ denote the event that
$$
\|A - \E A\| \leq K \sqrt{n}.
$$
As $A - \E A$ has centered, subgaussian entries, it is well known that
$$
\P(\EE_K^c) \leq e^{-c_p n}.
$$
Since $A - \E A$ is a mean-zero random matrix, the previous arguments apply to this matrix. The intuition is now to apply a covering argument to the range of $p J$, which is a low-dimensional subspace and therefore will not require many elements to construct an epsilon net.  We demonstrate this argument in its entirety for compressible vectors.  

\begin{lemma} \label{lem:complowrank}
	There exist constants $a, b, c_{\ref{lem:complowrank}} \in  (0, 1)$ and $K > 2$ such that for $\lambda \in \C^n$ with $|\lambda| \leq K \sqrt{n}$,
	$$
	\P\left(\inf_{\Bz \in \Comp(a,b)} \|(A - \lambda) \Bz\| \leq c_{\ref{lem:complowrank}} \sqrt{n} \text{ and } \EE_K \right) \leq e^{-c_{\ref{lem:complowrank}} n}.
	$$ 
\end{lemma}

\begin{proof}
	Let $\Bu \in \C^n$.  Since $A - \E A$ is mean-zero, we have that 
	\begin{equation}\label{eq:compshift}
	\P\left(\inf_{\Bz \in \Comp(a,b)} \|(A - \E A - (\lambda + pI)) \Bz - \Bu\| \leq c \sqrt{n} \text{ and } \EE_K \right) \leq e^{-c n}.
	\end{equation}
	Note that we have added a shift by a fixed vector $u$.  This version is well known and can be found in \cite[Proposition 4.2]{vershynin2014symmetric}.
	Now, let $\mathcal{N}$ be an $(c/2)\sqrt{n}$-net of $\{t (1, \dots, 1)^{\mathrm{T}}: t \in [-n, n]\}$ of size at most $C \sqrt{n}$.  On the event that there is a $\Bz \in \Comp(a,b)$ such that $\|(A - \lambda) \Bz\| \leq  (c/2) \sqrt{n}$, we must have that
 	for $\Bz' \in \mathcal{N}$ such that $\|\Bz' - J \Bz\| \leq (c/2) \sqrt{n}$, 
 	\begin{align*}
 	\|(A - \E A - (\lambda + p) \Bz + \Bz'\| &\leq \|(A - \E A - (\lambda + p) \Bz + J\Bz - (J\Bz - \Bz')\| \\
 	&\leq \|(A - \lambda)\Bz \| + \|J\Bz - \Bz'\| \\
 	&\leq c \sqrt{n}.
 	\end{align*}
 	By a union bound, the above event happens with probability at most $|\mathcal{N}'| e^{- c n} \leq e^{c' n}$. 	
\end{proof}

This trick of discretizing the range of $J$ can be applied to all the covering arguments from the previous section.  We leave the details to the reader.  Note that in the analogous covering argument for the complex disk, we still require that the complex shifts to be of norm at most $K \sqrt{n}$.  This allows us to conclude that eigenvectors of $A$ with corresponding eigenvalues in that disk have no arithmetic structure.  
The analogous multi-scale argument then allows us to conclude the following.
\begin{theorem} \label{thm:adjvec}
	There exist constants $C_{\ref{thm:adjvec}}, c_{\ref{thm:adjvec}} > 0 $ depending only on $p$ such that with probability at least $1 - e^{-c_{\ref{thm:adjvec}} n}$, for all eigenvectors $\Bv$ of $A$ corresponding to eigenvalues $\lambda$ such that $|\lambda| \leq 2 \sqrt{n}$, we have for $t \geq 0$,
	\begin{align*}
	\rho(\Bv, t) \leq C_{\ref{thm:adjvec}} t  +  e^{-c_{\ref{thm:adjvec}} n}.
	\end{align*}
\end{theorem}

One small complication that we have glossed over is that since we allow the possibility that the adjacency matrix be defined with zero diagonal, not all the entries are iid.  It is easy to show that this does not alter the argument much.
We show that removing a single coordinate of a vector cannot alter the LCD significantly.  

\begin{lemma}
	Let $\Bz \in \Incomp_{\C}(a,b)$ and $\Bz'$ be the vector $\Bz$ with any coordinate set to zero.  There exists a constant $c > 0$ such that   
	$$
	\rho(\Bz', t) \geq  c \rho(\Bz, t).
	$$
\end{lemma}

\begin{proof}
	Recall Definition \ref{def:LCD}.  
	By incompressibility, $\|\Bz'\| \geq b$ and $D(\Bz; L, 1) \geq c \sqrt{n}$.  Let $D = D(\Bz'/\|\Bz'\|; 2L, 1/4)$.  Fix $\eps > 0$.  There exists $\boldsymbol{\theta}$ such that $D \leq \|\boldsymbol{\theta}\| \leq D+\eps$ and
	$$
	\dist(U(\Bz'/\|\Bz'\|)^{\mathrm{T}} \boldsymbol{\theta}, \Z^{n-1} ) < \frac{1}{4} (2L) \sqrt{ \log_+ \frac{\|U(\Bz'/\|\Bz'\|)^{\mathrm{T}} \boldsymbol{\theta}\|}{2 L}}.
	$$   
	Note that we must have $\|U(\Bz'/\|\Bz'\| \boldsymbol{\theta}\| > 2L$.   
	We have
	\begin{align*}
	\dist\left(U(\Bz)^{\mathrm{T}} \frac{ \boldsymbol{\theta}}{\|\Bz'\|}, \Z^n\right) &\leq \dist \left(U(\Bz')^{\mathrm{T}} \frac{\boldsymbol{\theta}}{\|\Bz'\|}, \Z^{n-1} \right) + 1 \\
	&= \dist(U(\Bz'/\|\Bz'\|)^{\mathrm{T}}  \boldsymbol{\theta}, \Z^{n-1}) + 1 \\
	&< \frac{1}{4} (2L) \sqrt{ \log_+ \frac{\|U(\Bz'/\|\Bz'\|)^{\mathrm{T}}  \boldsymbol{\theta}\|}{2L}} + 1 \\
	&\leq \frac{1}{4} (2L) \sqrt{ \log_+ \frac{\|U(\Bz'/\|\Bz'\|)^{\mathrm{T}}  \boldsymbol{\theta}\|}{L}} + 1\\
	&\leq  L \sqrt{ \log_+ \frac{\|U(\Bz)^{\mathrm{T}}  \boldsymbol{\theta}/\|\Bz'\|\|}{L}}. \\
	\end{align*}
	As this is true for any $\eps > 0$, we therefore have
	$$
	\frac{D(\Bz'/\|\Bz'\|; 2L, 1/4)}{b} \geq \frac{D(\Bz'/\|\Bz'\|; 2L, 1/4)}{\|\Bz'\|}  > D(\Bz; L, 1) 
	$$ 
	Applying Theorem \ref{thm:smallballcomplex} completes the proof.
\end{proof}

Having established this, we leave it as an exercise to verify that all the structural results follow with only a slight change in the constants. 
\begin{remark} \label{remark:zerodiagonal}
	Using this same technique, all the structural results in Section \ref{sec:eigstruc} can be extended to random matrices with zero diagonal.  We omit the obvious modifications.
\end{remark} 

Although the above structural results only apply to eigenvectors with corresponding eigenvalues in the centered disk of radius $K \sqrt{n}$ in the complex plane, it is known that with high probability, this disk contains all the eigenvalues of $A$ but one.  In other words, our structural results apply to all eigenvectors but one with high probability.  
\begin{theorem}[Follows from Theorem 2.8 in \cite{OR}] \label{thm:outlier}
	Let $N$ be an iid random matrix whose entries are centered and have unit variance and finite fourth moment. Let $\tilde{N}$ be the matrix $N$ with the diagonal entries replaced with zeros.  Then for any $p \in (0, 1)$ and any $\delta > 0$, with probability $1-o(1)$, all the eigenvalues of $\tilde{N} + p J$ and $N + pJ$ are contained in the disk $\{z \in \mathbb{C} : |z| \leq (1 + \delta) \sqrt{n} \}$ with a single exception which takes the value $p n + o(\sqrt{n})$.  
\end{theorem}

To deduce some structural properties for the eigenvector of the lone eigenvalue outside the disk, we use the Perron--Frobenius theorem.  Recall the following definition.
\begin{definition}
	A square matrix $A$ is \emph{reducible} if there exists exists a permutation matrix $P$ such that 
	$$
	P^{\mathrm{T}} A P = \left( \begin{array}{cc}
	B & 0 \\
	C & D \end{array} \right)
	$$
	where $B$ and $D$ are square matrices.  A matrix is \emph{irreducible} if it is not reducble.
\end{definition}
\begin{remark}
	If $A$ is the adjacency matrix of a directed graph, irreducibility corresponds to strong-connectivity of the graph.  
\end{remark}

\begin{theorem}\cite[Theorem 6.8]{zhan2013matrixtheory} \label{thm:perronfrobenius}
	(Perron--Frobenius Theorem) If $A$ is an irreducible, nonnegative matrix then the eigenvalue with the largest norm is real and simple.  Furthermore, the eigenvector corresponding to this eigenvector has entries that are all strictly positive.
\end{theorem}

\begin{theorem}\cite[Theorem 5]{graham2008connectivity} \label{thm:stronglyconnected}
	In the directed random graph model with constant edge probability $ p \in (0,1)$, the graph is strongly connected with probability $1 - o(1)$.  
\end{theorem}

\section{Structure of Scaled Eigenvectors} \label{sec:scaledvectors}
\subsection{Eigenvector structure} 

The structure of eigenvectors from the previous section do not immediately apply to the Hadamard product of our eigenvectors with a fixed vector.  There are two issues that need to be overcome in this section.  The first is to deal with the possibly inhomogeneous values of the entries of $\Bb$.  In other words, although we have shown that any eigenvector $\Bu$ has no arithmetic structure, to handle the most general version, we must show that for our fixed vector $\Bb$, $\Bb \odot \Bu$ has no arithmetic structure. Here, we recall that $\Bb \odot \Bu$ denotes the Hadamard product of $\Bb$ and $\Bu$.  The second difficulty is that there is a small set of uncontrolled coordinates in $\Bb$.  In this section, we demonstrate how to deduce our main theorem from the arguments in the eigenvector structure theorem, but this requires repeating most of the steps from the previous section so we only sketch the argument here.

By absorbing the error probabilities of both Lemma \ref{lem:comp} and Lemma \ref{lem:realcomp} into our final error bound, we can assume that our approximate null-vectors are incompressible and have incompressibe real part. 

We recall the following condition on our fixed complex vector $\Bb$.
\begin{definition} \label{def:Bnudelocalized} Let $B \geq 1$ be a constant. 
	We say our vector $\Bb$ is \emph{$(B, m)$-delocalized} if we have:
	\begin{equation} \label{eq:coordinates}
	B^{-1} \leq |b_i| \leq B
	\end{equation}
	for all but $m$ entries of $\Bb$
\end{definition}

This is a more general definition than that used in \cite{OT2} as we do not require the entries to be rational. 

\begin{definition}
	Let $n_0 = n-m$.
	For $\Bu \in \C^n$, we let $\BBu \in \C^{n_0}$ denote the vector formed by the first $n_0$ coordinates of $\Bu$.  
	We define the function $F$ which takes $\BBu \in \C^{n_0}$ to 
	$$
	F(\BBu) = \underline{\Bb} \odot \underline{\Bu} = (u_1 b_1, \dots, u_{n_0} b_{n_0}). 
	$$
	We restrict our attention to those vectors in $\C^n$ that have no zero coordinates.  This poses no difficulty as we can infinitesimally shift any of our net points to avoid this measure-zero set.   
	Therefore, on this slightly restricted domain, our mapping $F$ is one-to-one so we can meaningfully speak of the inverse map $F^{-1}$.
\end{definition}

Without loss of generality, we assume that the first $n_0 = n-m$ entries of $\Bb$ satisfy (\ref{eq:coordinates}).
Therefore, we can assume that 
\begin{equation} \label{eq:normFtransform}
b B^{-1} \leq \|F(\BBu)\| = \sqrt{ \sum_{i=1}^{n_0} |u_i b_i|^2} \leq B
\end{equation}
where the first inequality follows from the incompressibility of $u$ and by assuming that $\nu$ is smaller than $a/2$, say, where we remind the reader that $\nu$ is in the statement of Theorem \ref{thm:maineigvecscaled} and $a$ is the constant from Lemma \ref{lem:comp}.  This assumption also guarantees that 
$$
\sum_{i=1}^{n_0} |u_i|^2 \geq b^2.
$$

Having established the notation, we briefly summarize the proof idea.  We condition on the event that all our potential eigenvectors lie in the unstructured subset of the sphere.
We consider the set $$F(U) = \{F(\Bu): \Bu \text{ a potential eigenvector}\}.$$
The goal is to show that for any eigenvector $\Bu$, $F(\Bu)$ has no arithmetic structure.  This is done with a similar covering argument as in Section \ref{sec:approximatenull}.  In fact, we have already constructed fine nets of the structured vectors on the unit sphere.  We then show that $F^{-1}$ maps this net to a fine net of our potential eigenvectors.  If there exists an eigenvector $\Bu \in U$ such that $F(\BBu)$ is structured, then there exists a vector $\Bv$ in our net such that $\Bv$ is structured (due to its proximity to $F(\Bu)$) and $F^{-1}(\Bv)$ is an approximate eigenvector since $F^{-1}(F(\Bv)) \approx \Bu$. This can be converted into a statement about being an approximate eigenvector by discretizing the possible eigenvalues and tensorizing as we have already seen.  Finally, the probability that $F^{-1}(\Bv)$ is an approximate eigenvector is small enough to survive the union bound over all possible $\Bv$ in our net.

Several subtleties have been overlooked in this description of our proof.  $F(\Bu)$ does not necessarily have norm 1, but typically this only adds a single dimension to our epsilon nets.  Additionally, our notion of structure actually encompasses several parameters (e.g. compressibility, LCD, real-imaginary correlation), so our argument needs to deal with these separately as in Section \ref{sec:approximatenull}.  Fortunately, many of the calculations can be recycled.  Due to the similarities, we will only provide full details for a few representative lemmas.

For now we fix a complex number $\lambda$ with $|\lambda| \leq K \sqrt{n}$.      

\begin{lemma} \label{lem:controlcomp}
	There exist constants $a', b' \in (0, 1)$ such that
	\begin{align*}
	\P\Big(\exists \Bu \in \Incomp_{\C}(a,b) \text{ such that } &\|(N - \lambda \sqrt{n}) \Bu\| \leq K b' \sqrt{n} \\
	&\text{ and } F(\BBu)/\|F(\BBu)\| \in \Comp(a', b')\Big) \leq e^{-c_{\ref{lem:controlcomp}} n}.
	\end{align*}
\end{lemma}

\begin{proof}
	We use $I$ to denote $\Incomp_{\C}(a, b)$.  We consider the event that $F(\BBu)/\|F(\BBu)\|_2 \in \Comp_{\C}(a', b')$.  By the standard volume argument, there exists a $b'/8T^2$-net, $\hat{\mathcal{N}}$ of $\Comp(a', b')$ of size at most 
	$$
	\binom{n_0}{a' n_0} (CB^{2}/b)^{2 a' n_0}  \leq \exp\big(a n_0 \log(e/a') + 2 a' n_0 \log(C B^{2}/ b')\big) . 
	$$
	By (\ref{eq:normFtransform}), it suffices to consider
	$$C = F(I) \cap \{t \cdot \Comp(a', b'): b' B^{-1} \leq t \leq B] \} .
	$$ 
	We use a union of discrete scalings of $\hat{\mathcal{N}}$ to create a net of $C$.
	Let 
	$$
	\mathcal{N} = \left \{t_j \cdot \hat{\mathcal{N}}: j \in [-8B^2/b', 8B^2/b'] \cap \Z \text{ and } t_j = \frac{j b'}{8B}  \right\}
	$$
	To see that this is a $b'/4B$-net of $C$, take $\Bw \in C$ and let $\Bv \in \hat{\mathcal{N}}$ such that $\left \| \frac{\Bw}{\|\Bw\|} - \Bv \right\| \leq b'/2B^2$.  Furthermore, let $t_j$ be such that $|t_j - \|\Bw\| | \leq b'/8B$.  Then
	$$
	\|\Bw - t_j \cdot \Bv\| = \|\Bw\| \left \| \frac{\Bw}{\|\Bw\|} - \Bv \right\| + \big|\|\Bw\| - t_j\big| \cdot \|\Bv\| \leq \frac{b'}{4B}.
	$$
	With a simple trick, we can modify $\mathcal{N}$ so that $\mathcal{N} \in C$ at the cost of changing $\mathcal{N}$ to a $b'/2B$-net.  The procedure is as follows.  For every $\Bv \in \mathcal{N}$, if there is an element of $C$ within a distance of $b'/8$, replace $\Bv$ with that element, choosing one arbitrarily if there are multiple options.  If there is no element of $C$ within $b'/8$, then remove $\Bv$ from $\mathcal{N}$.  It is easy to verify that this modified $\mathcal{N}$ is a $b'/2B$ net of $C$ and is of size at most 
	$$
	CB^2/b' \times \exp\big(a' n_0 \log(e/a') + 2 a n_0 \log(C B^{2}/ b')\big) \leq \exp\big(a' n_0 \log(e/a') + 3 a' n_0 \log(C B^{2}/ b') \big).
	$$
	
	We claim that $F^{-1}(\mathcal{N})$ is a $b'$-net of the set of vectors in $\Incomp_{\C}(a, b)$ such that $F(\BBu)/\|F(\BBu)\| \in \Comp_{\C}(a', b')$.  Consider a $\Bv$ such that $\|F(\BBu) - \Bv\| \leq b'/2B$.  Then,
	\begin{align*}
	\|\BBu - F^{-1}(\Bv)\| &= \sqrt{ \sum_{i=1}^{n_0} \left(u_i - \frac{ v_i}{f_i} \right)^2} \\
	&\leq \sqrt{ B^2 \sum_{i=1}^{n_0} \left(f_i u_i - v_i \right)^2} \\
	&\leq B  \|F(\Bu) - \Bv\| \\
	&\leq b'/2.
	\end{align*}
	We use $F^{-1}(\mathcal{N})$ to approximate the first $n_0$ coordinates.  We combine this with a simple volume net.  There exists a $b'/2$ net, $\mathcal{N}'$, of $K \cdot B_{m}(0)$ (where $B_{m}(0)$ is the unit ball in $\C^{m}$) of size at most $(CK/b')^{2 m}$. We define our final net 
	$$
	\mathcal{N}'' = \left \{ (\Bv, \Bv'): \Bv \in \mathcal{N}, \Bv' \in \mathcal{N"}' \right\}
	$$ 
	which is of size at most 
	\begin{align*}
	\exp\big(a' n_0 \log(e/a') + &3 a' n_0 \log(C B^{2}/ b') \big) \times (C/b')^{2\nu n} \\
	&\leq \exp\big(a' n_0 \log(e/a') + 3 a' n_0 \log(C B^{2}/ b') + 2 m \log(CK/b')\big).
	\end{align*}
	By the triangle inequality, $\mathcal{N}''$ is a $b$-net of the eigenvectors $\Bu$.  Therefore, since $\|(N - \lambda) \Bu\| \leq K b' \sqrt{n}$, 
	$$
	\|(N  - \lambda ) \Bv\| \leq K b' \sqrt{n} + \|N - \lambda\| \|\Bu - \Bv\|  \leq 3 K b \sqrt{n}.
	$$
	On the other hand, by a standard tensorization argument (c.f.	\cite[Lemma 3.2]{rudelson2009smallest}), for any $\Bv \in \mathcal{N}''$, 
	$$
	\P(\|(N - \lambda) \Bv\| \leq 3K b' \sqrt{n}) \leq e^{-c' n}
	$$
	for small enough $b'$.  Thus, by a union bound, 
	\begin{align*}
	\P(\exists \Bv \in \mathcal{N}'' \text{ such that } &\|(N - \lambda ) \Bv\| \leq 3K b' \sqrt{n}) \leq e^{-c' n} \times \\
	& \exp\big(a' n_0 \log(e/a') + 3 a' n_0 \log(C B^{2}/ b') + 2 m \log(CK/b')\big) \\
	&\leq \exp\big(a' n_0 \log(e/a') + 3 a' n_0 \log(C B^{2}/ b') + 2 m \log(CK/b') - c n\big) \\
	&\leq \exp(-c' n)
	\end{align*}
	where the last line follows from choosing $a', b'$ small enough and noting that $m = o(n)$.
\end{proof}

The same approximation procedure yields analogues of all the lemmas in Section \ref{sec:approximatenull}.  We illustrate this with one more example.  
\begin{proposition}
	Let $D \in [c_{\ref{lem:LCDincomp}} \sqrt{n}/ \alpha, D_0]$ be such that $m \leq \nu n/\log D$.  Recall the definition of $ S_{D, d, \alpha}$ in Definition \ref{def:complexreal}. 
	$$
	\P\left(\exists \Bu \in \Incomp_{\C}(a, b) \text{ s.t. } \|M \Bu\| \leq \frac{K \mu n}{2 D} \text{, } F(\BBu)/\|F(\BBu)\|_2 \in S_{D, d, \alpha}  \text{ and }  \EE_K \right) \leq e^{-c_{\ref{thm:levelcomplex}} n}.
	$$
\end{proposition}
\begin{proof}
	The proof follows the same strategy as the previous proof.  We generate a net of the set $F(\BBu)$ such that $\Bu \in \Incomp_{\C}(a, b)$ with $F(\BBu)/\|F(\BBu)\| \in S_{D, d, \alpha}$. By Proposition \ref{prop:complexnet}, there exists a $\mu \sqrt{n}/ D$-net, $\mathcal{N}'$, of $S_{D, d, \alpha}$ of size at most
	$$
	\frac{C_{\ref{prop:complexnet}}^{2n} D^{2n+1} d^{n-1}}{\mu^{n_0+1} n_0^{n_0 + 1/2}}.
	$$
	Therefore, we define a net that is composed of discrete scalings of $\mathcal{N}'$.  Let
	$$
	\mathcal{N} = \left\{t_j \cdot \mathcal{N}': t_j = \frac{j \mu \sqrt{n}}{2 D} \text{ for } j \in \N \text{ such that } t_j \in [bB^{-1}, B]\right \}.
	$$
	Observe that 
	$$
	|\mathcal{N}| \leq \frac{CDB}{\mu \sqrt{n}} \times \frac{C_{\ref{prop:complexnet}}^{2n_0} D^{2n_0+1} d^{n_0-1}}{\mu^{n_0+1} n_0^{n_0 + 1/2}}.
	$$
	We use a trivial net to estimate the remaining $m$ coordinates.  There is a $\mu \sqrt{n}/ 2D$-net of size at most $(CBD/\mu \sqrt{n})^{2 m}$ for $K \cdot B_{m}(0)$.  We combine this with $F^{-1}(\mathcal{N})$ to create a $4 \mu \sqrt{n}/D$-net of those approximate null-vectors with $F(\BBu)/\|F(\BBu)\|_2 \in S_{D, d, \alpha}$.  We call this net $\hat{\mathcal{N}}$.
	For any vector $u$ such that $F(\BBu)/\|F(\BBu)\| \in S_{D, d, \alpha}$ and $\|M \Bu\| \leq K \mu n/2 D$, there exists a $\Bv \in \hat{\mathcal{N}}$ from our net such that 
	$$
	\|M \Bv\| \leq K \mu n/D.  
	$$
	By Theorem \ref{thm:levelcomplex} and the proof therein, 
	\begin{align*}
	\P(\exists \Bv \in \hat{\mathcal{N}} \text{ s.t. } &\|M \Bv\| \leq K \mu n/D) \leq \sum_{\Bv \in \hat{\mathcal{N}}} \P(\|M \Bv\| \leq K \mu n/D) \\
	&\leq \frac{CDB}{\mu \sqrt{n}}  \frac{C_{\ref{prop:complexnet}}^{2n_0} D^{2n_0+1} d^{n_0-1}}{\mu^{n+1} n_0^{n_0 + 1/2}}  (CBD/\mu \sqrt{n})^{2 m} \times \frac{2^n C_{\ref{prop:tensorcomplex}}^{n} L^{2n} (2 K \mu \sqrt{n})^{2 n}}{d^{n} D^{2 n}} \\
	&\leq B^2 C^n \mu^n  d_0^{-m} \\
	&\leq B^2 C^n \mu^n ( D)^{m} \\
	&\leq B^2 C^n \mu^n \exp(\nu n) \\
	&\leq \exp(-c n).
	\end{align*}
	The small-ball probability follows from Proposition \ref{prop:tensorcomplex} and Proposition \ref{thm:levelcomplex}.  The third to last inequality is the crucial line that determines the trade-off between $m$ and $D$.  
\end{proof}

Combining the analogous propositions and lemmas yield the analogous strucutre theorem for approximate null-vectors.  Finally, to conclude the same structure theorem for eigenvectors, we use the approximation argument from Section \ref{sec:eigstruc}.  Ultimately, this leads to the following structural theorem.

\begin{theorem} \label{thm:maincontrol}
	Fix a constant $B \geq 1$.  There exist constants $c_\star, \nu, C_{\ref{thm:maincontrol}}, c_{\ref{thm:maincontrol}}$ possibly depending on $B$ such that the following holds.
	Let $c_{\ref{thm:maincontrol}} \sqrt{n} \leq D \leq e^{c_\star n}$ and $m \in \N$ such that $m \leq \nu n/\log D$.  For a $(B, m)$-delocalized vector $\Bb$, 
	$$
	\P\left(\exists \text{ eigenvector } \Bv \text{ of } N \text{ such that } \rho(\Bb \odot \Bv, t) \geq C_{\ref{thm:maincontrol}}\left(t + \frac{1}{D} + \frac{t^2 D}{\sqrt{n}} \right) \text{ and } \EE_K \right) \leq e^{-c_{\ref{thm:maincontrol}} n}
	$$
\end{theorem}

We provide one specific choice of $m$ and $t$ to demonstrate possible consequences of this theorem.
\begin{corollary}
	Fix a constant $B \geq 1$.  Then for any constant $c < 1$ and a fixed $(B, n^{1-c})$-delocalized vector $\Bb$, 
	$$
	\P\left(\exists \text{ eigenvector } \Bv \text{ of } N \text{ such that } \rho(\Bb \odot \Bv, 0) \leq C_{\ref{thm:maincontrol}} \exp(-n^c) \text{ and } \EE_K \right) \leq e^{-c_{\ref{thm:maincontrol}} n}
	$$
\end{corollary}

An identical series of theorems can be proved for the adjacency matrix case using the approximation techniques of Section \ref{sec:eigstruc}.

Again, we would like to extend the range of effective bounds by combining our bounds at different scales as we did at the end of Section \ref{sec:eigstruc}.  Due to the dependence of $m$ on $D$, we will have an extra complication.  

\begin{corollary} \label{cor:alldscaled}
	We fix a $B \geq 1$, $D' \in [c_{\ref{thm:maincontrol}} \sqrt{n}, e^{c_\star n}]$ and $m \leq \nu n/ \log D'$.  For a $(B, m)$-delocalized  vector $\Bb$,
	\begin{align*}
	\P\Bigg(\exists \text{ eigenvector } \Bv \text{ of } &N, \, D \in [c_{\ref{thm:maincontrol}} \sqrt{n}, D'] \text{ and } t \geq 0\\
	&\text{ such that } \rho(\Bb \odot \Bv, t) \geq C_{\ref{thm:maincontrol}}\left(t + \frac{1}{D} + \frac{t^2 D}{\sqrt{n}} \right) \text{ and } \EE_K \Bigg) \leq e^{-c_{\ref{thm:maincontrol}} n}.
	\end{align*}
\end{corollary}
\begin{proof}
	Let $d_k = c_{\ref{thm:maincontrol}} \sqrt{n} 2^k$.  By our choice of $m$, for any $d_k$ with $k \in \N$ such that $d_k \in [c_{\ref{thm:maincontrol}} \sqrt{n}, D']$, we can apply Theorem \ref{thm:maincontrol} to conclude that with probability at least $1 - e^{-c_{\ref{thm:maincontrol}} n}$, for $\Bb$ a $(B,m)$-delocalized vector and for any eigenvector $\Bv$ of $N$ will be such that for $t \geq 0$,
	\begin{align*}
		\rho(\Bb \odot \Bv, t) &\leq C_{\ref{thm:maincontrol}}\left(t + \frac{1}{d_k} + \frac{t^2 d_k}{\sqrt{n}} \right).
	\end{align*}
	On this event, for any $D \in [d_k, d_{k+1})$, 
	\begin{align*}
	\rho(\Bb \odot \Bv, t) &\leq C_{\ref{thm:maincontrol}}\left(t + \frac{1}{d_k} + \frac{t^2 d_k}{\sqrt{n}} \right) \\
	&\leq 2C_{\ref{thm:maincontrol}}\left(t + \frac{1}{D} + \frac{t^2 D}{\sqrt{n}} \right).
	\end{align*}
	Taking a union bound over $k \in \N$ with $d_k \in [c_{\ref{thm:maincontrol}} \sqrt{n}, D']$ concludes the proof.
\end{proof}

Now, we allow $D$ to vary with $t$ to boost our result to all scales.  

\begin{theorem} \label{thm:smallballscaledeigvec}
	We fix a $B \geq 1$.  There exist constants $\nu, \nu', C_{\ref{thm:smallballscaledeigvec}}, c_{\ref{thm:smallballscaledeigvec}} > 0$ possibly depending on $B$ such that for $m \leq \nu \sqrt{n}$ and a $(B, m)$-delocalized  vector $\Bb$,
	\begin{align*}
	\P\Bigg(\exists \text{ eigenvector } \Bv \text{ of } &N  \text{ and } t \geq e^{- \nu' n/m}\\
	&\text{ such that } \rho(\Bb \odot \Bv, t) \geq C_{\ref{thm:maincontrol}}t \text{ and } \EE_K \Bigg) \leq e^{-c_{\ref{thm:maincontrol}} n}.
	\end{align*}
\end{theorem}

\begin{proof}
	This results follows from applying Corollary \ref{cor:alldscaled} with $D' = e^{\nu n/m}$, $D = \sqrt{n}/t$ and restricting $t$ so that $m \leq \nu n/ \log D$ as required in Corollary \ref{cor:alldscaled}.
\end{proof}

\section{Completing the Proofs and Deducing Controllability} \label{sec:controllability}

This section is devoted to the proofs of our main results and their corollaries.  The key tool is the following proposition.

\begin{proposition} \label{prop:reduce}
Let $N$ be an iid matrix with symmetric atom variable $\xi$ that satisfies Assumption \ref{assump:main}.  Fix constants $B, K \geq 1$.  Then there exist positive constants $c_\star, \nu, C_{\ref{prop:reduce}}, c_{\ref{prop:reduce}}$ depending on $B, K$, and $\xi$ such that the following holds.
Let $m \leq \nu \sqrt{n}$.  For a $(B, m)$-delocalized vector $\Bb \in \mathbb{C}^n$ and for any $t \geq e^{-\nu' n/m}$,
\begin{align*}
	\P( \exists \text{ a unit eigenvector } \Bv \text{ of } N \text{ such that } | \Bb^{\mathrm{T}} \Bv| \leq t ) \leq C_{\ref{prop:reduce}} n t  + \Prob( \mathcal{E}_K^c). 
\end{align*}
\end{proposition} 

\begin{remark} \label{remark:zerodiagonalreduce}
	The above proposition also applies when the matrix $N$ is an iid matrix except with zeros along the diagonal. This is a simple consequence of Remark \ref{remark:zerodiagonal} and the proof of Proposition \ref{prop:reduce}.  We omit the details. 
\end{remark}

We prove Proposition \ref{prop:reduce} in Section \ref{sec:reduce:proof} below.  Theorem \ref{thm:maineigvecscaled} follows immediately from Proposition \ref{prop:reduce}.   
Theorem \ref{thm:maineigvec} is a consequence of Theorem \ref{thm:maineigvecscaled} since the all-ones vector is $(B, 0)$-delocalized for any $B \geq 1$.

\subsection{Controllability} 
While Definition \ref{def:kalman} gives Kalman's rank condition for the pair $(A,\Bb)$ to be controllable, it is not the most useful criteria to check.  Instead, in this section, we will focus on the Popov--Belevitch--Hautus (PBH) test.  This test was introduced independently by Popov \cite{Popov}, Belevitch \cite{Belevitch}, Hautus \cite{Hautus}, Rosenbrock \cite{Rosenbrock}, Hahn \cite[p. 27]{Kalman}, Johnson \cite{FJ1}, Ford and Johnson \cite{FJ2}, and Gilbert \cite{Gilbert}.  The version presented below appears as Theorem 2.4-8 in \cite{Kls}.  

\begin{theorem} [PBH eigenvector test] \label{thm:PBH}
The pair $(A, \Bb)$ is uncontrollable if and only if there exists a left eigenvector $\Bv$ of $A$ such that $\Bv^\mathrm{T} \Bb = 0$.  
\end{theorem}

In order to study the probability that $(A, \Bb)$ is controllable, the PBH test allows us to study the probability that a left eigenvector of $A$ is orthogonal to $\Bb$.  In fact, if $A$ is an iid matrix (or the adjacency matrix of a directed Erd\H{o}s--R\'enyi random graph), $A$ and $A^\mathrm{T}$ have the same distribution, and it suffices to study the probability that a (right) eigenvector of $A$ is orthogonal to $\Bb$.  In order to do so, we will apply Proposition \ref{prop:reduce}.

In view of Theorem \ref{thm:PBH}, by taking $t$ as small as possible, Proposition \ref{prop:reduce} allows us to bound the probability that $(A, \Bb)$ is uncontrollable.  Indeed, we immediately obtain the following corollary for an iid matrix.  

\begin{corollary} \label{cor:reduce}
Let $N$ be an iid matrix with symmetric atom variable $\xi$ that satisfies Assumption \ref{assump:main}.  Fix constants $B, K \geq 1$.  Then there exist positive constants $\nu, C_{\ref{cor:reduce}}, c_{\ref{cor:reduce}}$ depending on $B, K$, and $\xi$ such that the following holds.  Let $m \leq \nu \sqrt{n}$.  For a $(B, m)$-delocalized vector $\Bb \in \mathbb{C}^n$, 
\[ \P( (N, \Bb) \text{ is uncontrollable} ) \leq C_{\ref{cor:reduce}}e^{-c_{\ref{cor:reduce}}n} + \P( \mathcal{E}_K^c). \]
\end{corollary} 

Corollary \ref{cor:allones} now follows immediately from Corollary \ref{cor:reduce} and \eqref{eq:4momnorm}.  

We finish this subsection with a proof of Corollary \ref{cor:allonesgraph}.  
\begin{proof}[Proof of Corollary \ref{cor:allonesgraph}]
Recall that $\1$ is the all-ones vector.  Let $B := A - \frac{1}{2} J$, where $J = \1 \1^{\mathrm{T}}$ is the all-ones matrix.  If $\Bv$ is an eigenvector of $A$ that is orthogonal to $\1$, then $\Bv$ must also be an eigenvector of $B$ (since $J \Bv = 0$).  

We will work with the matrix $N := 2 (B + \frac{1}{2} I)$, where $I$ is the identity matrix.  The matrix $N$ has the same eigenvectors as $B$ (since shifting by a multiple of the identity matrix and scalar multiplication do not change the eigenvectors), and the entries of $N$ are iid Rademacher random variables, except for the diagonal entries which are identically zero.  

By Proposition \ref{prop:reduce} and Remark \ref{remark:zerodiagonalreduce}, $N$ is uncontrollable with probability at most $Ce^{-c n}$ for some $C,c > 0$ since the entries of $N$ are subgaussian.  Hence, $B$ is uncontrollable with the same probability.  From the controllability of $B$ we can conclude the controllability of $A$ due to following chain of implications: 
\begin{align*}
(A, \boldsymbol{1}) \text{ is uncontrollable} &\Longleftrightarrow \exists \lambda, v \text{ such that } v \neq 0, Av = \lambda v \text{ and } \boldsymbol{1}^{\mathrm{T}} v = 0 \\
&\Longleftrightarrow \exists \lambda, v \text{ such that } v \neq 0, (A - \frac{1}{2} \boldsymbol{1} \boldsymbol{1}^{\mathrm{T}}) v = \lambda v \text{ and } \boldsymbol{1}^{\mathrm{T}} v = 0 \\
&\Longleftrightarrow (B, \boldsymbol{1}) \text{ is uncontrollable},
\end{align*}  
which completes the proof.
\end{proof}

\subsection{Random Vectors: Proofs of Corollaries \ref{cor:random} and \ref{cor:randomA}}

In order to prove Corollaries \ref{cor:random} and \ref{cor:randomA}, we will need the following lemma.  

\begin{lemma} \label{lemma:randomvec}
Let $\xi$ be a real-valued random variable with mean zero, unit variance, and finite fourth moment. Let $N$ be the $n \times n$ iid random matrix with atom variable $\xi$.  Let $\psi$ be a real-valued random variable that satisfies Assumption \ref{assump:main}, and assume $\Bb \in \mathbb{R}^n$ is a random vector with entries that are iid copies of $\psi$.  Then
\[ \P( \exists \text{ a unit eigenvector } \Bv \text{ of } N \text{ such that } \Bb^{\mathrm{T}} \Bv = 0 ) = o(1). \]
\end{lemma}
\begin{proof}
In view of \eqref{eq:4momnorm}, it follows that there exists a constant $K > 1$ so that $\mathcal{E}_K$ holds with probability $1 - o(1)$.  
We say the eigenvalues of $N$ are simple if $N$ has $n$ distinct eigenvalues (each with multiplicity one).  Let $\mathcal{S}$ denote the event that the eigenvalues of $N$ are simple.  It follows from Theorem \ref{thm:eiggaptail} that $\mathcal{S}$ holds with probability $1 - o(1)$.  

Let $\EE$ denote the event that there exists a unit eigenvector $\Bv$ of $N$ with $\rho(\Bv, 0) > e^{-c_{\ref{cor:rhozero}} n}$.  It follows from Corollary \ref{cor:rhozero} that
\[ \P( \EE) \leq \P (\EE \cap \EE_K) + \P(\EE_K^c) = o(1). \]
Therefore, we conclude that
\begin{align*}
	\P( \exists &\text{ a unit eigenvector } \Bv \text{ of } N \text{ such that } \Bb^{\mathrm{T}} \Bv = 0 ) \\
	&\leq \P( \exists \text{ a unit eigenvector } \Bv \text{ of } N \text{ such that } \Bb^{\mathrm{T}} \Bv = 0 | \EE^c \cap \mathcal{S} ) + o(1).
\end{align*}
On the event $\mathcal{S}$, $N$ has $n$ distinct eigenvectors, determined uniquely up to sign.  Let $\Bv_1, \ldots, \Bv_n$ denote the unit eigenvectors of $N$ on the event $\mathcal{S}$.  Since the choice of sign for each eigenvector does not effect whether $\Bb^{\mathrm{T}} \Bv_i$ is zero or not, we adopt the convention that each eigenvector $\Bv_i$ is multiplied by a random sign, independent of all other sources of randomness.  
We obtain
\begin{align*}
	\P( \exists &\text{ a unit eigenvector } \Bv \text{ of } N \text{ such that } \Bb^{\mathrm{T}} \Bv = 0 | \EE^c \cap \mathcal{S} ) \\
	&\leq \P( \exists i \in [n] \text{ such that } \Bb^{\mathrm{T}} \Bv_i = 0 | \EE^c \cap \mathcal{S}). 
\end{align*}
On the event $\EE^c$, $\rho(v_i, 0) \leq e^{-c_{\ref{cor:rhozero}} n}$ for all $i \in [n]$.  So by the union bound, 
\[ \P( \exists i \in [n] \text{ such that } \Bb^{\mathrm{T}} \Bv_i = 0 | \EE^c \cap \mathcal{S}) \leq n e^{-c_{\ref{cor:rhozero}} n} = o(1). \]
The proof of the lemma is complete.  
\end{proof}

Corollary \ref{cor:random} now follows from Lemma \ref{lemma:randomvec} and Theorem \ref{thm:PBH}.  
Similarly Corollary \ref{cor:randomA} follows from the following lemma.  

\begin{lemma} \label{lemma:randomvecA}
Let $A$ be the $n \times n$ adjacency matrix of an Erd\H{o}s--R\'enyi directed graph with constant edge probability $p \in (0,1)$.  Let $\psi$ be a real-valued random variable that satisfies Assumption \ref{assump:main}, and assume $\Bb \in \mathbb{R}^n$ is a random vector with entries that are iid copies of $\psi$.   Then
\[ \P( \exists \text{ a unit eigenvector } \Bv \text{ of } A \text{ such that } \Bb^{\mathrm{T}} \Bv = 0 ) = o(1). \]
\end{lemma}
\begin{proof}
The argument is identical to the proof of Lemma \ref{lemma:randomvec} except for the following changes:
\begin{itemize}
\item One must use Theorem \ref{thm:adjeiggaptail} instead of Theorem \ref{thm:eiggaptail}.  
\item Instead of Corollary \ref{cor:rhozero}  one needs to apply Theorem \ref{thm:adjvec}.  
\item It only remains to address the eigenvector, $\Bv$, associated to the largest eigenvalue.  By Theorems \ref{thm:perronfrobenius} and \ref{thm:stronglyconnected}, the eigenvector is entirely positive, which in particular implies that each entry is non-zero.  We now appeal to an anti-concentration inequality which is a generalization of the classical result of Erd\H{o}s-Littlewood-Offord. 
\begin{lemma}[L\'evy-Kolmogorov-Rogozin, \cite{rogozin1961smallball}] \label{lem:rogozin}
	Let $\xi_i$ be independent real-valued random variables.  Then for any non-negative real numbers $r_1, \dots, r_m$ and $r \geq \max_i(r_i)$,
	$$
	\rho\left(\sum_{i=1}^m \xi, r\right) \leq \frac{C_{\ref{lem:rogozin}} r}{\sqrt{\sum_{i=1}^m (1 - \rho(\xi_i, r_i)) r_i^2}}
	$$
	for a universal constant $C_{\ref{lem:rogozin}} > 0$.  
\end{lemma}   
Therefore, applying the lemma with $r = r_i = \min_i v_i > 0$, 
$$
\P(\Bb^{\mathrm{T}} \Bv = 0) \leq \frac{C}{\sqrt{n}}, 
$$ 
where $C$ only depends on $p$.
\end{itemize}
\end{proof}

\subsection{Minimal Controllability}
Our eigenvector structure results can quickly lead to a result on minimal controllability.  
\begin{proof}[Proof of Corollary \ref{cor:basis}]
By symmetry it suffices to bound the probability that $(N, \Be_1)$ is controllable.  
In view of \eqref{eq:4momnorm} it suffices to upper bound
\[ \P ( (N, \Be_1) \text{ is uncontrollable and } \mathcal{E}_K ) \]
for some sufficiently large constant $K > 1$.  

Let us decompose our matrix $N$ as
$$
N = \left( \begin{array}{cc}
 N_{11} & \BX^{\mathrm{T}} \\
 \BY & N'
\end{array} \right)
$$
where $N_{11}$ denotes the $(1,1)$-entry of $N$, $\BX, \BY \in \R^{n}$ and $N'$ is an $(n-1) \times (n-1)$ matrix.  Moreover, $N_{11}$, $\BX$, $\BY$, and $N'$ are jointly independent.  
Using Theorem \ref{thm:PBH}, we need to upper bound the probability that a unit eigenvector of $N$ is orthogonal to $\Be_1$.  The key observation is that if there exists a unit eigenvector 
$$
\Bv = \left(\begin{array}{c} v_{1} \\
\Bv' \end{array} \right)
$$
 that is orthogonal to $\Be_1$  then $\Bv'$ is a unit eigenvector of $N'$ and $\BX^{\mathrm{T}} \Bv' = 0$.  Thus, it suffices to show that
 \begin{equation} \label{eq:randomvec}
	\P ( \exists \text{ a unit eigenvector } \Bv \text{ of } N' \text{ such that } \BX^{\mathrm{T}} \Bv = 0 ) = o (1). 
\end{equation}
Since the entries of $\BX$ are iid random variables, independent of $N'$, and satisfy Assumption \ref{assump:main}, the bound in \eqref{eq:randomvec} follows from Lemma \ref{lemma:randomvec}; the proof is complete.   
\end{proof}

The proof of Corollary \ref{cor:basisgraph} follows from a nearly identical argument.
\begin{proof}[Proof of Corollary \ref{cor:basisgraph}]
	The proof is identical to that of Corollary \ref{cor:basis} except for the following points.
	\begin{itemize}
		\item We use Lemma \ref{lemma:randomvecA} instead of Lemma \ref{lemma:randomvec} to handle all eigenvectors with eigenvalues within a disk of radius $2 \sqrt{n}$.
		\item By Theorem \ref{thm:outlier}, it only remains to address the eigenvector, $\Bv$, associated to the largest eigenvalue.  By Theorems \ref{thm:perronfrobenius} and \ref{thm:stronglyconnected}, the eigenvector is entirely positive with high probability which, in particular, implies that $\Be_i^{\mathrm{T}}  \Bv \neq 0$.  
	\end{itemize}
\end{proof}

\subsection{Proof of Proposition \ref{prop:reduce}} \label{sec:reduce:proof}
This section is devoted to the proof of Proposition \ref{prop:reduce}.  The main idea is to utilize the symmetry of the atom distribution of $N$ to rewrite the dot product $\Bb^{\mathrm{T}} \Bv$ as a small ball probability (in particular, conditioned on the matrix $N$, we rewrite the dot product as a sum of independent random variables).  The same idea was exploited in \cite{OT2} to study the controllability of real symmetric random matrices.  

\begin{proof}[Proof of Proposition \ref{prop:reduce}]
Let $\boldsymbol{\xi} = (\eps_1, \ldots, \eps_n)$, where $\eps_1, \ldots, \eps_n$ are iid Rademacher random variables, independent of $N$, i.e., each $\eps_i$ takes the values $\pm 1$ with probability $1/2$.  
We say the eigenvalues of $N$ are simple if $N$ has $n$ distinct eigenvalues (each with multiplicity one).   Let $\mathcal{S}$ be the event that the eigenvalues of $N$ are simple and that $\mathcal{E}_K$ holds.  We have
\begin{align*} 
	\P( &\exists \text{ a unit eigenvector } \Bv \text{ of } N \text{ such that } | \Bb^{\mathrm{T}} \Bv| \leq t ) \\
	&\leq \P( \exists \text{ a unit eigenvector } \Bv \text{ of } N \text{ such that } | \Bb^{\mathrm{T}} \Bv| \leq t \text{ and } \mathcal{S} ) + \P(\mathcal{S}^c),
\end{align*}
and by Theorem \ref{thm:eiggaptail}
\[ \P(\mathcal{S}^c) \leq Ce^{-cn} + \Prob( \mathcal{E}_K^c). \]

We now turn our attention to bounding
\[ \P( \exists \text{ a unit eigenvector } \Bv \text{ of } N \text{ such that } | \Bb^{\mathrm{T}} \Bv| \leq t \text{ and } \mathcal{S} ). \]
On the event $\mathcal{S}$, $N$ has $n$ distinct unit eigenvectors $\Bv_1, \ldots, \Bv_n$, which are determined uniquely up to sign.  As the choice of sign does not change the value of $| \Bb^{\mathrm{T}} \Bv_i|$, we will simply assume that each eigenvector is multiplied by a random sign, independent of all other sources of randomness.  Then
\begin{align}
	\P( &\exists \text{ a unit eigenvector } \Bv \text{ of } N \text{ such that } | \Bb^{\mathrm{T}} \Bv| \leq t \text{ and } \mathcal{S} ) \nonumber
	\\&\leq \P( \exists i \in [n] \text{ such that } | \Bb^{\mathrm{T}} \Bv_i| \leq t \text{ and } \mathcal{S} ). \label{eq:unionbnde} 
\end{align}
We can now exploit the fact that the entries of $N = (N_{ij})_{i,j=1}^n$ are symmetric random variables.  Indeed, let $N' = (\eps_i \eps_j N_{ij})_{i,j=1}^n$.  A simple calculation shows that the eigenvalues of $N'$ are the same as the eigenvalues of $N$.  In addition, when $\Bv_1, \ldots, \Bv_n$ are the eigenvectors of $N$, then $\Bv_1 \odot \boldsymbol{\xi}, \ldots, \Bv_n \odot \boldsymbol{\xi}$ are the eigenvectors of $N'$.  Here, $\Bu \odot \Bv$ denotes the Hadamard product of the vectors $\Bu = (u_i)$ and $\Bv = (v_i)$ defined by $\Bu \odot \Bv = ( u_i v_i)$.  
Since the atom variable of $N$ is symmetric, it follows that $N'$ is an iid matrix and that $N'$ has the same distribution as $N$.  This implies that the eigenvectors $\Bv_1 \odot \boldsymbol{\xi}, \ldots, \Bv_n \odot \boldsymbol{\xi}$ have the same distribution as $\Bv_1, \ldots, \Bv_n$.  Hence, we conclude that
\[ \P(\exists i \in [n] \text{ such that } | \Bb^{\mathrm{T}} \Bv_i| \leq t \text{ and } \mathcal{S} ) = \P( \exists i \in [n] \text{ such that }| \Bb^{\mathrm{T}} (\Bv_i \odot \boldsymbol{\xi}) | \leq t \text{ and } \mathcal{S} ). \]
The probability that $| \Bb^{\mathrm{T}} (\Bv_i \odot \boldsymbol{\xi}) | \leq t$ can be bounded above by the small ball probability $\rho( \Bb \odot \Bv_i, t)$, and so we can now apply Theorem \ref{thm:smallballscaledeigvec}.  Indeed, Theorem \ref{thm:smallballscaledeigvec} guarantees the existence of an event $\EE$, which holds with probability at least $1 - O(e^{-c_{\ref{thm:smallballscaledeigvec}} n})$, so that conditioned on this event the eigenvectors $\Bv_1, \ldots, \Bv_n$ are such that
\begin{equation} \label{eq:sbpvec}
	\sup_{1 \leq i \leq n } \rho( \Bb \odot \Bv_i, t) \leq C_{\ref{thm:smallballscaledeigvec}} t. 
\end{equation} 
Returning to \eqref{eq:unionbnde}, it suffices to bound
\[ 
	\P( \exists i \in [n] \text{ such that }| \Bb^{\mathrm{T}} (\Bv_i \odot \boldsymbol{\xi}) | \leq t | \mathcal{S} \cap \EE).
\]
 Applying the union bound and \eqref{eq:sbpvec} yields the desired conclusion.   
\end{proof}

\bibliographystyle{abbrv}
\bibliography{bib}

\appendix
\section{Tail Bounds on Eigenvalue Gaps} \label{appendix:tailtwoeigs}
In this section we prove Theorem \ref{thm:eiggaps}.  We follow the approach in \cite{ge2017eigenvalue} and include some details for the reader's convenience.  Additionally, we fix several oversights in \cite{ge2017eigenvalue} along the way.  

\subsection{Reduction from Eigenvalues to Singular Values}
The following lemma is the first step in converting the eigenvalue problem into one of singular values, which are more stable and amenable to approximation arguments.
\begin{lemma}
	Let $N \in \C^{n \times n}$ with $\|N\| \leq K \sqrt{n}$, $z \in \C$ with $|z| \leq K $ and denote $M = N - z \sqrt{n}$.  
	Suppose there exist $i,j \in [n]$ such that the eigenvalues of $N$, $\lambda_i, \lambda_j \in B(z \sqrt{n}, s)$ for $0< s \leq K \sqrt{n}$.  Then there exist orthogonal vectors $\Bv, \Bw \in S_{\C}^{n-1}$ and a real number $\alpha$ with $|\alpha| \leq 2K \sqrt{n}$ such that
	\begin{equation}\label{eq:orthdecomp}
	M \Bv = (\lambda_i - z\sqrt{n}) \Bv \text{ and } M\Bw = (\lambda_j -z \sqrt{n}) w + \alpha \Bv.
	\end{equation}
	As a consequence, 
	\begin{equation} \label{eq:normbound}
	\|M \Bv\| \leq s \text{ and } \|M \Bw - \alpha \Bv\| \leq s
	\end{equation}  
\end{lemma}

\begin{proof}
	We begin with the assumption that $\lambda_i \neq \lambda_j$ and let $\Bv_i$ and $\Bv_j$ be two corresponding eigenvectors.  Then we can choose $\Bv = \Bv_i$.  We will choose $\Bw$ to be orthogonal to $\Bv$ and also in the span of $\Bv_i$ and $\Bv_j$.  Let us write $\Bw = a_i \Bv_i + a_j \Bv_j$.  Therefore,
	\begin{equation} 
	M w = a_i (\lambda_i - z\sqrt{n}) v_i + a_j (\lambda_j - z \sqrt{n}) v_j = (\lambda_j - z\sqrt{n}) w + a_i(\lambda_i - \lambda_j) v_i = (\lambda_j - z \sqrt{n}) w + \alpha v
	\end{equation}
	where $ \alpha = a_i (\lambda_i - \lambda_j)$.  Since $\alpha v = (M - (\lambda_j - z \sqrt{n})) w$, 
	$$
	|\alpha| = \|\alpha v\| \leq 2 K \sqrt{n}.
	$$
	Furthermore, 
	$$
	\|Mw - \alpha v\| = \|(\lambda_j - z \sqrt{n}) w\| \leq s.
	$$
	
	If $\lambda_i = \lambda_j$, but the geometric multiplicity is greater than or equal to two, then the above argument still applies since we can find distinct eigenvectors $v_i \neq v_j$.  Thus, the only remaining case is when $\lambda = \lambda_i = \lambda_j$ and the geometric multiplicity of $\lambda$ is one.  By the Jordan canonical form, there exist $v_i \neq v_j$ such that
	$$
	M v_i = (\lambda - z \sqrt{n}) v_i \text{ and } M v_j = (\lambda v_j + v_i) - z \sqrt{n} v_j = (\lambda - z \sqrt{n}) v_j + v_i.
	$$  
	Using the notation $w = a_i v_i + a_j v_j$ for a vector orthogonal to $v_i$, we have
	$$
	M w = (\lambda - z \sqrt{n}) w + a_j v_i
	$$
	so we can use $\alpha = a_j$ and complete the proof as above.
\end{proof}

The next lemma allows us to consider bounding the tails of the least singular value and and the second smallest singular value.  

\begin{lemma} \label{lem:eigstosings}
	Let $N \in \C^{n \times n}$, $z \in \C$ with $|z| \leq K $ and denote $M = N - z \sqrt{n}$.  
	Suppose there exist $i,j \in [n]$ such that the eigenvalues of $N$, $\lambda_i, \lambda_j \in B(z \sqrt{n}, s)$ for $s > 0$. 
	\begin{enumerate}
		\item If $\alpha \in \C$ with $|\alpha| \leq s$ that satisfies (\ref{eq:orthdecomp}), then
		$$
		s_n(M) \leq s \text{ and } s_{n-1}(M) \leq 2s.
		$$   
		\item If $\alpha \in \C$ with $|\alpha| > s$ that satisfies (\ref{eq:orthdecomp}), then
		$$
		s_n(M) \leq \frac{s^2}{|\alpha|} \text{ and } s_{n-1}(M) \leq |\alpha|.
		$$   
	\end{enumerate}   
	
\end{lemma}
\begin{proof}
	We begin with the first case.  
	We have $\|M \Bv\| \leq s$ and 
	$$
	\|M \Bw\| \leq \|M \Bw - \alpha \Bv\| + \| \alpha \Bv\| \leq 2s
	$$
	As $\Bv$ and $\Bw$  are orthogonal, we have $s_n(M) \leq s$ and $s_{n-1}(M) \leq 2s$.  
	
	Now we assume that $|\alpha| > s$.  Since
	$$
	\|M|_{\text{span}(\Bv, \Bw)} \| \leq 2s,
	$$
	we have $s_{n-1}(N) \leq 2s$.  Also,
	$$
	s_n(M) \leq s_2(M|_{\text{span}(\Bv, \Bw)} ) \leq \dist(M \Bv, \text{span}(M \Bw)).
	$$
	To evaluate the right-hand side of this inequality, we recall that 
	$$
	M\Bv = (\lambda_i - z \sqrt{n}) \Bv \text{ and } M\Bw = (\lambda_j -z \sqrt{n}) w + \alpha \Bv
	$$
	which implies 
	$$
	\dist(M \Bv, \text{span}(M \Bw)) \leq \frac{|\lambda_i - z\sqrt{n}||\lambda_j - z \sqrt{n}|}{|\alpha|} \leq \frac{s^2}{|\alpha|}.
	$$
\end{proof}

\subsection{Two Smallest Singular Values}
In this section, we consider a fixed complex shift $\lambda$ with imaginary part $ \delta \geq e^{-c_{\star}n}$ and we prove the next proposition.
\begin{proposition}
	For any $t' \geq t \geq 0$, there exists constants $C, c > 0$ such that
	$$
	\P\left(s_n(N - \lambda \sqrt{n}) \leq \frac{t}{\sqrt{n}},\, s_{n-1}(N - \lambda \sqrt{n}) \leq \frac{t'}{\sqrt{n}} \text{ and } \EE_K\right) \leq C \frac{t^2 t'^2}{\delta^2} + e^{-c n}
	$$
\end{proposition}

\begin{definition}
	We say a subspace $W$ is \emph{incompressible} if all $v \in S_{n-1} \cap W$ are in $\Incomp_{\C}(a,b)$.
\end{definition}
	
	Let $v$ and $v'$ be the right singular vectors corresponding to $s_{n}(N - \lambda \sqrt{n})$ and $s_{n-1}(N - \lambda \sqrt{n})$.  Let $W$ denote the subspace spanned by $v$ and $v'$, then on the event $s_{n-1}(N - \lambda \sqrt{n}) \leq \frac{t'}{\sqrt{n}}$, 
	$$
	\|(A - \lambda \sqrt{n})|_W \| \leq \frac{t'}{\sqrt{n}}.
	$$
	Therefore, 
	\begin{align*}
	\P\Big(s_n(A - &\lambda \sqrt{n}) \leq \frac{t}{\sqrt{n}},\, s_{n-1}(N - \lambda \sqrt{n}) \leq \frac{t'}{\sqrt{n}} \text{ and } \EE_K\Big) \\ & \leq \P\left(\inf_{z \in \Incomp} \|(N - \lambda \sqrt{n}) z\|\leq \frac{t}{\sqrt{n}}, \, \inf_{W \in \Incomp} \|(N - \lambda \sqrt{n})|_W\| \leq \frac{t'}{\sqrt{n}} \text{ and } \EE_K \right) \\
	&\qquad + \P\Big(\inf_{z \in \Comp} \|(N - \lambda \sqrt{n}) z\| \leq \frac{t'}{\sqrt{n}} \text{ and } \EE_K\Big)
	\end{align*}
	The last term is exponentially small by Lemma \ref{lem:comp}.  The next lemma converts the remaining probability into a distance problem.

\begin{lemma}[Lemma A.1.4, \cite{ge2017eigenvalue}]
	Let $M = N - \lambda \sqrt{n}$.  
	\begin{align*}
	\P(\inf_{z \in \Incomp(a,b)} \|M z\|< \frac{b t}{\sqrt{n}}, \, &\inf_{W \in \Incomp} \|M|_{W}\| < \frac{b t'}{\sqrt{n}} \text{ and } \EE) \\
	&\leq \frac{1}{a^2 n^2} \sum_{k=1}^n \sum_{j \neq k} \P(\dist(X_k, H_k) < t, \dist(X_j, H_{jk}) < t' \text{ and } \EE)
	\end{align*}
	where $X_k$ denotes the $k$-th row of $M$, $H_k$ is the span of all the rows except the $k$-th and $H_{jk}$ is the span of all the rows except the $j$-th and $k$-th.  
\end{lemma}

The next two propositions yield tail bounds on the distance problems depending on whether the complex shift is real or not. 
\begin{proposition} \label{prop:twosingvalues}
	There exist constants $C, c > 0$ such that
	for $\lambda \in \C$ with $|\lambda| \leq K \sqrt{n}$ and $\delta = \im(\lambda) \geq e^{c_* n}$.  
	$$
	\P(s_n(N - \lambda \sqrt{n}) \leq \frac{t}{\sqrt{n}}, \, s_{n-1}(N - \lambda \sqrt{n}) \leq \frac{t'}{\sqrt{n}} \text{ and } \EE_K) \leq C \frac{t^2 t'^2}{\delta^2} + e^{-c n}.
	$$
\end{proposition}

\begin{proof}
	By Proposition \ref{prop:twosingvalues} and symmetry, we can focus on 
	$$
	\P(\dist(X_n, H_n)< t, \dist(X_{n-1}, H_{n,n-1}) < t' \text{ and } \EE_K).
	$$
	The event that $\dist(X_n, H_n)$ implies that there exists a unit vector $X$ orthogonal to $H_n$ such that $|\langle X, X_n \rangle| < t$.  Similarly, the event that $\dist(X_{n-1}, H_{n, n-1}) < t'$ implies that for all vectors $X'$ orthogonal to $H_{n,n-1}$ with $|\langle X', X_{n-1} \rangle| < t'$.  By Theorem \ref{thm:nullvectors} and Theorem \ref{thm:smallballcomplex}, with probability at least $1 - e^{-c n}$, for any vector, $\Bv$, orthogonal to $H_n$, 
	$$
	\rho(\Bv, t) \leq \frac{C}{\delta} \left(t + e^{-cn} \right)^2.
	$$
	We denote this event by $\EE$.
	One can easily check that the proof of Theorem \ref{thm:nullvectors} applies equally well to vectors orthogonal to $H_{n,n-1}$ so we also have that any vector $\Bx$ orthogonal to $H_{n, n-1}$, with probability at least $1 - e^{-c n}$, 
	$$
	\rho(\Bx, t) \leq \frac{C}{\delta} (t + e^{-c n})^2.
	$$
	We call this event $\EE'$. 
	Therefore, 
	\begin{align*}
	\P(&\dist(X_n, H_n) < t, \dist(X_{n-1}, H_{n,n-1}) < t' \text{ and } \EE_K)  \\
	&\leq \P(|\langle X, X_n \rangle| < t, \EE_K \text{ and } \EE) \P(|\langle X', X_{n-1} \rangle| < t, \EE_K \text{ and } \EE') + \P(\EE^c) + \P(\EE'^c) \\
	&\leq  \left(\frac{C}{\delta} (t + e^{-c n})^2 \right) \left(\frac{C}{\delta} (t' + e^{-c n})^2 \right) + 2 e^{-cn}
	\end{align*}
	where the last line follows from the independence of $X_n$, $X_{n-1}$ and $H_{n, n-1}$.  The result follows after reducing  $c_{\star}$ if necessary.  
\end{proof}

Finally, we recall a tail bound for \emph{real} shifts.  
\begin{proposition}[Theorem 3.2.5 \cite{ge2017eigenvalue}] \label{prop:twosingvaulesreal}
	There exist constants $C, c > 0$ such that for any $t' \geq t \geq 0$, and \emph{real} $\lambda$ with $|\lambda| \leq K \sqrt{n}$ then
	$$
	\P(s_n(N - \lambda \sqrt{n}) \leq \frac{t}{\sqrt{n}}, \, s_{n-1}(N - \lambda \sqrt{n}) \leq \frac{t'}{\sqrt{n}} \text{ and } \EE_K) \leq C (t t') + e^{-c n}.
	$$
\end{proposition}

\subsection{Tail Bounds on Gaps}
\begin{proposition} \label{prop:doubleeig}
	
$$
\P(\exists \lambda_i, \lambda_j \in B(z, \delta) \text{ and } \EE_K) \leq n^2 \delta^2 + e^{-cn}
$$	

$$
\P(\exists \lambda_i, \lambda_j \in B(z, s) \text{ and } \EE_K) \leq C \log(n/s)\frac{s^4 n^4}{\delta^2} + e^{-c n}
$$
\end{proposition}

\begin{proof}
	We define $\alpha_1 = s$ and recursively, $\alpha_k = 2^k s$.  For any $\alpha \in \C$ with $\alpha_{k} \leq |\alpha| < \alpha_{k+1}$ and $s < |\alpha|$.  Then the event that $s_n(M) \leq \frac{(s/\sqrt{2})^2}{|\alpha_{k}|}$ and $s_{n-1}(M) \leq |\alpha_{k}|$ implies that
	$$
	s_n(M) \leq \frac{s^2}{|\alpha|} \text{ and } s_{n-1}(M) \leq |\alpha|.
	$$    
	Thus, by Lemma \ref{lem:eigstosings},
	\begin{align*}
	\P( \lambda_i, \lambda_j \in B(z, s) \text{ and } \EE_K) &\leq \P(s_n(M) \leq s_{n-1}(N) \leq 2s) \\
	&\qquad + \sum_{k =1}^{C \log(n/s)} \P(s_n(M) \leq (s/\sqrt{2})^2/|\alpha_{k}| \text{ and } s_{n-1}(M) \leq |\alpha_{k}|)
	\end{align*}
	where the range of the sum is determined by the condition that $|\alpha| \leq 2K \sqrt{n}$.   For every summand,
	Proposition \ref{prop:twosingvalues} provides an upper bound of $C (s/\sqrt{2})^4 n^2/\delta^2 + e^{-c n}$.  Taking a union bound over the choice of $i,j \in [n]$ concludes the proof of the first statement.  
	
	An analogous argument using Proposition \ref{prop:twosingvaulesreal} instead of Proposition \ref{prop:twosingvalues} yields the second result.  
\end{proof}

\subsection{Proof of Theorem \ref{thm:eiggaptail}}

\begin{proof}
Let $D_K$ denote the disk of radius $K \sqrt{n}$ in the complex plane.  We begin with a $\delta/10$-net, $\mathcal{N}$, of the intersection of the real line with $D_K$.  Such a net can be constructed to be of size less than $2K \sqrt{n} \delta^{-1}$.  We center a ball of radius $\delta$ on each point in the net.  The union of these balls contains a strip of size $\delta/2$ around the section of the real line in $D_K$.  Let $D'_K$ be $D_K$ after removing a strip of width $\delta/100$ around the real line.  We can construct an $s/10$-net, $\mathcal{N'}$, of $D'_K$ of size at most $C K^2 s^{-2} n$.  If $s < \delta/2$, then on the event that there exist $\lambda_i, \lambda_j$ such that $|\lambda_i - \lambda_j| \leq s$, we must have either for some $z_k \in \mathcal{N}$  
$$
\lambda_i, \lambda_j \in B(z_k, \delta)
$$     
or for some $z'_k \in \mathcal{N}'$
$$
\lambda_i, \lambda_j \in B(z'_k, s).
$$
Both these events are controlled in Proposition \ref{prop:doubleeig}.  Thus, by a union bound,
\begin{align*}
\P(\Delta \leq s \text{ and } \EE_K) &\leq \sum_{z_k \in \mathcal{N}} \P(\exists \lambda_i, \lambda_j \in B(z_k, \delta) \text{ and } \EE_K) \\
&\qquad + \sum_{z_k' \in \mathcal{N}'} \P(\exists \lambda_i, \lambda_j \in B(z_k', s) \text{ and } \EE_K) \\
&\leq C K \sqrt{n} \delta^{-1} (n^2 \delta^2) + C K^2 s^{-2} n (\log(n/s) \frac{s^4 n^4}{\delta^2}) + 2 e^{-cn}\\
&\leq C K n^{5/2} \delta + C K^2 \frac{s^2 n^5}{\delta^2} \log(n/s) + 2 e^{-c n}
\end{align*}
Observe that this bound is only effective in the range $\delta \leq n^{-5/2}$.  In this range, if we set $\delta = c s^{2/3} n^{5/6}$ we have $s < \delta/2$ for a small enough constant $c$.  Then we can conclude that
$$
\P(\Delta \leq s \text{ and } \EE_K) \leq (C K^2 c^{-2}) s^{2/3} n^{10/3} \log(n/s)  + e^{-c n}
$$ 
Finally, to simplify the result, we generously bound $\log(n/s)$ by $n$ using the fact that $\delta \geq e^{-c^* n}$.
\end{proof}

\section{Tail bounds for Eigenvalue Gaps of Adjacency Matrices} \label{appendix:adjtailtwoeigs}
In this section we sketch the necessary modifications to handle the gap probability for adjacency matrices.  We recall the basic structure of the spectrum.

\begin{theorem}[Follows from Theorem 2.8 in \cite{OR}] 
Let $N_n$ be an iid random matrix whose entries are centered and have unit variance and finite fourth moment. Let $\tilde{N}_n$ be the matrix $N_n$ with the diagonal entries replaced with zeros.  Then for any $p \in (0, 1)$ and any $\delta > 0$, almost surely, for $n$ sufficiently large, all the eigenvalues of $\tilde{N}_n + p J$ are contained in the disk $\{z \in \mathbb{C} : |z| \leq (1 + \delta) \sqrt{n} \}$ with a single exception which takes the value $p n + o(\sqrt{n})$.  
\end{theorem}

Due to the previous result, as the outlier eigenvalue is significantly separated from the others, it suffices to consider those eigenvalues within a radius of $K \sqrt{n}$ of the origin.  

\begin{lemma} \label{lem:amatrixtwoeigs}
	Let $A \in \C^{n \times n}$ with $\|A\| \leq 2n$, $z \in \C$ with $|z| \leq K \sqrt{n} $ and denote $M = A - z $.  
	Suppose there exist $i,j \in [n]$ such that the eigenvalues of $A$, $\lambda_i, \lambda_j \in B(z, s)$ for $0< s \leq 2n$.  Then there exist orthogonal vectors $v, w \in S_{\C}^{n-1}$ and a real number $\alpha$ with $|\alpha| \leq 2n$ such that
	\begin{equation}
	M \Bv = (\lambda_i - z) \Bv \text{ and } M\Bw = (\lambda_j -z ) \Bw + \alpha \Bv.
	\end{equation}
	As a consequence, 
	\begin{equation}
	\|M \Bv\| \leq s \text{ and } \|M \Bw - \alpha \Bv\| \leq s
	\end{equation}  
\end{lemma}

\begin{proof}
	We begin with the assumption that $\lambda_i \neq \lambda_j$ and let $\Bv_i$ and $\Bv_j$ be two corresponding eigenvectors.  Then we can choose $\Bv = \Bv_i$.  We will choose $\Bw$ to be orthogonal to $\Bv$ and also in the span of $\Bv_i$ and $\Bv_j$.  Let us write $\Bw = a_i \Bv_i + a_j \Bv_j$.  Therefore,
	\begin{equation} 
	M \Bw = a_i (\lambda_i - z) \Bv_i + a_j (\lambda_j - z) \Bv_j = (\lambda_j - z) \Bw + a_i(\lambda_i - \lambda_j) \Bv_i = (\lambda_j - z ) \Bw + \alpha \Bv
	\end{equation}
	where $ \alpha = a_i (\lambda_i - \lambda_j)$.  Since $\alpha \Bv = (M - (\lambda_j - z)) \Bw$, 
	$$
	|\alpha| = \|\alpha v\| \leq 4 n.
	$$
	Furthermore, 
	$$
	\|M\Bw - \alpha \Bv\| = \|(\lambda_j - z) \Bw\| \leq s.
	$$
	
	If $\lambda_i = \lambda_j$, but the geometric multiplicity is greater than or equal to two, then the above argument still applies since we can find distinct eigenvectors $\Bv_i \neq \Bv_j$.  Thus, the only remaining case is when $\lambda = \lambda_i = \lambda_j$ and the geometric multiplicity of $\lambda$ is one.  By the Jordan canonical form, there exist $\Bv_i \neq \Bv_j$ such that
	$$
	M \Bv_i = (\lambda - z ) \Bv_i \text{ and } M \Bv_j = (\lambda \Bv_j + \Bv_i) - z \Bv_j = (\lambda - z ) \Bv_j + \Bv_i.
	$$  
	Using the notation $\Bw = a_i \Bv_i + a_j \Bv_j$ for a vector orthogonal to $\Bv_i$, we have
	$$
	M \Bw = (\lambda - z ) \Bw + a_j \Bv_i
	$$
	so we can use $\alpha = a_j$ and complete the proof as above.
\end{proof}

The remainder of the argument is identical to that in Appendix \ref{appendix:tailtwoeigs}.  Finally, to control the distance problem, we utilize Theorem \ref{thm:adjvec} instead of Theorem \ref{thm:nullvectors}.  The reader can easily check that the norm of $A$ or $A - z$ does not appear in the argument outside of Lemma \ref{lem:amatrixtwoeigs} and Theorem \ref{thm:adjvec}.  It is in the proof of Theorem \ref{thm:adjvec} that we have overcome the majority of the large norm issues.

\end{document}